\documentclass{amsart}

\newtheorem{theorem}{Theorem}[section]
\newtheorem{lemma}[theorem]{Lemma}
\newtheorem{proposition}[theorem]{Proposition}
\newtheorem{corollary}[theorem]{Corollary}

\theoremstyle{remark}
\newtheorem{remark}[theorem]{Remark}

\numberwithin{equation}{section}

\def\r{{\mathbb R}}
\def\dsp{\displaystyle}

\def\eqref#1{(\ref{#1})}
\def\1{{1\hspace{-1.2mm}{\rm I}}}
\def\F{\mathcal F}
\def\kxi{\begin{pmatrix}1\\ \xi\end{pmatrix}}
\def\cal{\mathcal}

\usepackage{color,soul}

\begin{document}
\title[Kinetic entropy and HR scheme for the Saint-Venant system]{Kinetic entropy inequality and hydrostatic reconstruction scheme for the Saint-Venant system}

\author[E. Audusse]{Emmanuel Audusse}
\address{Universit\'e Paris 13, Laboratoire d'Analyse, G\'eom\'etrie et Applications, 
99 av. J.-B. Cl\'ement, F-93430 Villetaneuse, France - Inria, ANGE
project-team, Rocquencourt - B.P. 105, F78153 Le Chesnay cedex, France
- CEREMA, ANGE project-team,  134 rue de Beauvais, F-60280
Margny-L\`es-Compi\`egne, France - Sorbonne University, UPMC University Paris VI, ANGE project-team, UMR 7958 LJLL, F-75005 Paris, France}
\email{eaudusse@yahoo.fr}
\author[F. Bouchut]{Fran\c cois Bouchut}
\address{Universit\'e Paris-Est, Laboratoire d'Analyse et de Math\'ematiques Appliqu\'ees
(UMR 8050), CNRS, UPEM, UPEC, F-77454, Marne-la-Vall\'ee, France}
\email{Francois.Bouchut@u-pem.fr}
\author[M.-O. Bristeau]{Marie-Odile Bristeau}
\address{Inria, ANGE
project-team, Rocquencourt - B.P. 105, F78153 Le Chesnay cedex, France
- CEREMA, ANGE project-team,  134 rue de Beauvais, F-60280
Margny-L\`es-Compi\`egne, France - Sorbonne University, UPMC University Paris VI, ANGE project-team, UMR 7958 LJLL, F-75005 Paris, France}
\email{Marie-Odile.Bristeau@inria.fr}
\author[J. Sainte-Marie]{Jacques Sainte-Marie}
\address{Inria, ANGE
project-team, Rocquencourt - B.P. 105, F78153 Le Chesnay cedex, France
- CEREMA, ANGE project-team,  134 rue de Beauvais, F-60280
Margny-L\`es-Compi\`egne, France - Sorbonne University, UPMC University Paris VI, ANGE project-team, UMR 7958 LJLL, F-75005 Paris, France}
\email{Jacques.Sainte-Marie@inria.fr}

\begin{abstract}
A lot of well-balanced schemes have been proposed for discretizing the classical Saint-Venant system
for shallow water flows with non-flat bottom. Among them,
the hydrostatic reconstruction scheme is a simple and efficient one.
It involves the knowledge of an arbitrary solver
for the homogeneous problem (for example Godunov, Roe, kinetic\ldots).
If this solver is entropy satisfying, then the hydrostatic reconstruction scheme
satisfies a semi-discrete entropy inequality.
In this paper we prove that, when used with the classical kinetic
solver, the hydrostatic reconstruction scheme also satisfies a
fully discrete entropy inequality, but with an error term.
This error term tends to zero strongly when the space step tends to zero, including solutions with shocks.
We prove also that the hydrostatic reconstruction scheme does not satisfy the entropy inequality
without error term.
\end{abstract}



\maketitle

\section{Introduction}

The classical Saint-Venant system for shallow water describes the height of
water $h(t,x)\geq 0$, and the water velocity $u(t,x)\in\r$ ($x$ denotes a coordinate
in the horizontal direction) in the direction parallel to the bottom.
It assumes a slowly varying topography $z(x)$, and reads
\begin{equation}\begin{array}{l}
	\dsp \partial_t h+\partial_x(hu)=0,\\
	\dsp\partial_t(hu)+\partial_x(hu^2+g\frac{h^2}{2})+gh\partial_x z=0,
	\label{eq:SV}
	\end{array}
\end{equation}
where $g>0$ is the gravity constant.
This system is completed with an entropy (energy) inequality
\begin{equation}
	\partial_t\biggl(h\frac{u^2}{2}+g\frac{h^2}{2}+ghz\biggr)
	+\partial_x\biggl(\bigl(h\frac{u^2}{2}+gh^2+ghz\bigr)u\biggr)\leq 0.
	\label{eq:entropySV}
\end{equation}
We shall denote $U=(h,hu)^T$ and
\begin{equation}
	\eta(U)=h\frac{u^2}{2}+g\frac{h^2}{2},
	\qquad G(U)=\bigl(h\frac{u^2}{2}+gh^2\bigr)u
	\label{eq:etaG}
\end{equation}
the entropy and entropy fluxes without topography.

The derivation of an efficient, robust and stable numerical scheme for
the Saint-Venant system has received an extensive coverage. The issue involves the
notion of well-balanced schemes, and we refer the reader to~\cite{bouchut_book,J,Gbook,XS}
and references therein.

The hydrostatic reconstruction (HR), introduced in \cite{ABBKP}, is a general and efficient method
that evaluates an arbitrary solver for the homogeneous problem, like Roe, relaxation, or kinetic solvers
on {\sl reconstructed states} built with the steady state relations.
It leads to a consistent, well-balanced, positive scheme satisfying a semi-discrete entropy inequality,
in the sense that the inequality holds only in the limit when the timestep tends to zero.
The method has been generalized to balance all subsonic steady-states in \cite{BM}, and to multi-layer
shallow water in \cite{BZ} with the source-centered variant of the hydrostatic reconstruction.
Generic extensions are provided in \cite{CPP}, and a case of moving water is treated in \cite{NX}.
The HR technique enables second-order computations on unstructured meshes, see \cite{bristeau}.
It has also been used to derive efficient and robust numerical schemes
approximating the incompressible Euler and Navier-Stokes equations
with free surface \cite{ABPerthSM,ABPelSM},  i.e. non necessarily shallow water flows.

The aim of this paper is to prove that the hydrostatic reconstruction,
when used with the classical kinetic solver~\cite{BGKcons,kinetic-RR,PS,FVS,bristeau,GSM,BGSM}, satisfies a
fully discrete entropy inequality, stated in Corollary \ref{corol:estim_lip_l1}.
However, as established in Proposition \ref{prop necessary-error},
this inequality necessarily involves an error term. The main result of this paper is that
this error term is in the square of the topography increment,
ensuring that it tends to zero strongly as the space step
tends to zero, for solutions that can include shocks. The topography needs however to be
Lipschitz continuous.

In general, to satisfy an entropy inequality is a criterion for the stability of a scheme.
In the fully discrete case, it enables in particular to get an {\it a priori} bound on the total energy.
In the time-only discrete case and without topography, the single energy inequality that holds for the kinetic scheme
ensures the convergence \cite{BB}. The fully discrete case (still without topography) has been
treated in \cite{Ber}. Another approach to get a scheme satisfying a fully discrete entropy inequality
is proposed in \cite{CSS}. Following our results, the proof of convergence of the hydrostatic reconstruction scheme
with kinetic numerical flux will be performed in a forthcoming paper.

The outline of the paper is as follows.
We recall in Section~\ref{sec:kin_without_topo} the
kinetic scheme without topography and its entropy analysis,
in both the discrete and semi-discrete cases. We show in particular how one can see that
the fully discrete inequality is always less dissipative than the semi-discrete one, see Lemma \ref{lemma:seminotopo}.
In Section~\ref{sec:kin_HR} we propose a kinetic interpretation of the hydrostatic reconstruction
and we give its properties. We analyze in detail the entropy inequality.
The semi-discrete scheme is considered first.
Our main result Theorem~\ref{thentrfull_JSM} concerning the fully discrete scheme is finally proved.\\

We end this section by recalling the classical kinetic approach, used in \cite{PS} for example,
and its relation with numerical schemes.
The kinetic Maxwellian is given by
\begin{equation}
	M(U,\xi)=\frac{1}{g\pi}\Bigl(2gh-(\xi-u)^2\Bigr)_+^{1/2},
	\label{eq:kinmaxw}
\end{equation}
where $U=(h,hu)^T$, $\xi\in\r$ and $x_+\equiv\max(0,x)$ for any $x\in\r$.
It satisfies the following moment relations,
\begin{equation}\begin{array}{c}
	\dsp \int_\r \kxi M(U,\xi)\,d\xi=U,\\
	\dsp\int_\r \xi^2 M(U,\xi)\,d\xi=hu^2+g\frac{h^2}{2}.
	\label{eq:kinmom}
	\end{array}
\end{equation}
These definitions allow us to obtain a {\sl kinetic representation} of the
Saint-Venant system.
\begin{lemma}
If the topography $z(x)$ is Lipschitz continuous,
the pair of functions $(h, hu)$ is a weak solution to the Saint-Venant system~\eqref{eq:SV}
if and only if $M(U,\xi)$ satisfies the kinetic equation
\begin{equation}
	\partial_t M+\xi\partial_x M-g(\partial_x z)\partial_\xi M=Q,
	\label{eq:kinrepres}
\end{equation}
for some ``collision term'' $Q(t,x,\xi)$ that satisfies, for a.e. $(t,x)$,
\begin{equation}
	\int_\r Q d\xi = \int_\r \xi Qd \xi = 0.
	\label{eq:kinrepresintcoll}
\end{equation}
\label{lemma:sv_kin_rep}
\end{lemma}
\begin{proof}
If \eqref{eq:kinrepres} and \eqref{eq:kinrepresintcoll} are satisfied,
we can multiply \eqref{eq:kinrepres} by $(1,\xi)^T$, and integrate with respect to $\xi$.
Using~\eqref{eq:kinmom} and \eqref{eq:kinrepresintcoll} and integrating by parts
the term in $\partial_\xi M$, we obtain \eqref{eq:SV}.
Conversely, if $(h, hu)$ is a weak solution to \eqref{eq:SV}, just define $Q$ by \eqref{eq:kinrepres};
it will satisfy \eqref{eq:kinrepresintcoll} according to the same computations.
\end{proof}
The standard way to use Lemma \ref{lemma:sv_kin_rep} is to write
a kinetic relaxation equation~\cite{Perthame90,Perthame92,perthame_coron,BGKcons,FVS}, like
\begin{equation}
	\partial_t f+\xi\partial_x f-g(\partial_x z)\partial_\xi f=\frac{M-f}{\epsilon},
	\label{eq:kinrelax}
\end{equation}
where $f(t,x,\xi)\geq 0$, $M=M(U,\xi)$ with $U(t,x)=\int(1,\xi)^Tf(t,x,\xi)d\xi$,
and $\epsilon>0$ is a relaxation time. In the limit $\epsilon\to 0$ we recover formally
the formulation \eqref{eq:kinrepres}, \eqref{eq:kinrepresintcoll}.
We refer to \cite{BGKcons} for general considerations on such kinetic relaxation models
without topography, the case with topography being introduced in \cite{PS}.
Note that the notion of {\sl kinetic representation} as \eqref{eq:kinrepres}, \eqref{eq:kinrepresintcoll}
differs from the so called {\sl kinetic formulations} where a large set of entropies is involved,
see \cite{Perthame}. For systems of conservation laws, these kinetic formulations
include non-advective terms that prevent from writing down simple approximations.
In general, kinetic relaxation approximations can be compatible with just a single entropy.
Nevertheless this is enough for proving the convergence as $\epsilon\to 0$, see \cite{BB}.\\

Apart from satisfying the moment relations \eqref{eq:kinmom}, the particular form
\eqref{eq:kinmaxw} of the Maxwellian is taken indeed for its compatibility with a kinetic entropy,
that ensures energy dissipation in the relaxation approximation \eqref{eq:kinrelax}.
Consider the kinetic entropy
\begin{equation}
	H(f,\xi,z)=\frac{\xi^2}{2}f+\frac{g^2\pi^2}{6}f^3+gzf,
	\label{eq:kinH}
\end{equation}
where $f\geq 0$, $\xi\in\r$ and $z\in\r$, and its version without topography
\begin{equation}
	H_0(f,\xi)=\frac{\xi^2}{2}f+\frac{g^2\pi^2}{6}f^3.
	\label{eq:kinH0}
\end{equation}
Then one can check the relations
\begin{equation}
	\int_\r H\bigl(M(U,\xi),\xi,z\bigr)\,d\xi=\eta(U)+ghz,
	\label{eq:kinint}
\end{equation}
\begin{equation}
	\int_\r \xi H\bigl(M(U,\xi),\xi,z\bigr)\,d\xi=G(U)+ghzu.
	\label{eq:kinflint}
\end{equation}
One has the following subdifferential inequality and entropy minimization principle.
\begin{lemma}
\noindent (i) For any $h\geq 0$, $u\in\r$, $f\geq 0$ and $\xi\in\r$
\begin{equation}
	H_0(f,\xi)\geq H_0\bigl(M(U,\xi),\xi\bigr)+\eta'(U)\kxi
	\bigl(f-M(U,\xi)\bigr).
	\label{eq:subdiff}
\end{equation}

\noindent (ii) For any $f(\xi)\geq 0$, setting $h=\int f(\xi)d\xi$, $hu=\int \xi f(\xi)d\xi$ (assumed finite), one has
\begin{equation}
	\eta(U)=\int_\r H_0\bigl(M(U,\xi),\xi\bigr)\,d\xi
	\leq\int_\r H_0\bigl(f(\xi),\xi\bigr)\,d\xi.
	\label{eq:entropmini}
\end{equation}
\end{lemma}
\begin{proof}
This approach by the subdifferential inequality has been introduced in \cite{BGKcons}.
The property (ii) easily follows from (i) by taking $f=f(\xi)$ and integrating \eqref{eq:subdiff}
with respect to $\xi$. For proving (i), notice first that
\begin{equation}
	\eta'(U)=\bigl(gh-u^2/2,u\bigr),
	\label{eq:etaprime}
\end{equation}
where prime denotes differentiation with respect to $U=(h,hu)^T$. Thus
\begin{equation}
	\eta'(U)\kxi
	=gh-u^2/2+\xi u=\frac{\xi^2}{2}+gh-\frac{(\xi-u)^2}{2}.
	\label{eq:etaprime1xi}
\end{equation}
Observe also that
\begin{equation}
	\partial_f H_0(f,\xi)=\frac{\xi^2}{2}+\frac{g^2\pi^2}{2}f^2.
	\label{eq:H0prime}
\end{equation}
The formula defining $M$ in \eqref{eq:kinmaxw} yields that
\begin{equation}
	gh-\frac{(\xi-u)^2}{2}=\left\{\begin{array}{l}
	\dsp \frac{g^2\pi^2}{2}M(U,\xi)^2\quad\mbox{if }M(U,\xi)>0,\\
	\dsp\mbox{is nonpositive\quad if }M(U,\xi)=0,
	\end{array}\right.
	\label{eq:idM}
\end{equation}
thus
\begin{equation}
	\partial_f H_0\bigl(M(U,\xi),\xi\bigr)=\left\{\begin{array}{l}
	\dsp \eta'(U)\kxi\quad\mbox{if }M(U,\xi)>0,\\
	\dsp\geq\eta'(U)\kxi \quad\mbox{if }M(U,\xi)=0.
	\end{array}\right.
	\label{eq:idH'}
\end{equation}
We conclude using the convexity of $H_0$ with respect to $f$ that
\begin{equation}\begin{array}{l}
	\dsp H_0(f,\xi)\geq H_0\bigl(M(U,\xi),\xi\bigr)+\partial_f H_0\bigl(M(U,\xi),\xi\bigr)\bigl(f-M(U,\xi)\bigr)\\
	\dsp \hphantom{H_0(f,\xi)}
	\geq H_0\bigl(M(U,\xi),\xi\bigr)+\eta'(U)\kxi
	\bigl(f-M(U,\xi)\bigr),
	\label{eq:convineqH_0}
	\end{array}
\end{equation}
which proves the claim.
\end{proof}

For numerical purposes it is usual to replace the right-hand side in the kinetic relaxation
equation \eqref{eq:kinrelax} by a time discrete projection to the Maxwellian state.
When space discretization is present it leads to flux-vector splitting schemes, see \cite{FVS}
for the case without topography, \cite{PS} for the case with
topography, and \cite{bristeau} for the 2d case on unstructured meshes.\\

Here we consider more general schemes. We would like to approximate the solution $U(t,x)$, $x \in \r$, $t
\geq 0$ of the system~\eqref{eq:SV} by discrete values $U_i^n$, $i \in
\mathbb{Z}$, $ n \in \mathbb{N}$. In order to do so, we consider a grid of points $x_{i+1/2}$, $i \in
\mathbb{Z}$,
$$\ldots < x_{i-1/2} < x_{i+1/2} <  x_{i+3/2} < \ldots,$$
and we define the cells (or finite volumes) and their lengths
$$C_i = ]x_{i-1/2},x_{i+1/2}[,\qquad \Delta x_i = x_{i+1/2} - x_{i-1/2}.$$
We consider discrete times $t^n$ with $t^{n+1}=t^n+\Delta t^n$, and
we define the piecewise constant functions $U^n(x)$ corresponding to time $t^n$ and $z(x)$ as
\begin{equation}
	U^n(x)= U^n_i,\quad z(x)= z_i,\quad\mbox{ for }x_{i-1/2}<x<x_{i+1/2}.
	\label{eq:U^npc}
\end{equation}
A finite volume scheme for solving~\eqref{eq:SV} is a formula of the form
\begin{equation}
	U^{n+1}_i=U^n_i-\sigma_i(F_{i+1/2-}-F_{i-1/2+}),
	\label{eq:upU0}
\end{equation}
where $\sigma_i=\Delta t^n/\Delta x_i$, telling how to compute the
values $U^{n+1}_i$ knowing $U_i^n$ and discretized values $z_i$ of the topography.
Here we consider first-order explicit three points schemes where
\begin{equation}
	F_{i+1/2-} = \F_l(U_i^n,U_{i+1}^n,z_{i+1}-z_i),\qquad F_{i+1/2+} = \F_r(U_i^n,U_{i+1}^n,z_{i+1}-z_i).
\label{eq:flux_def}
\end{equation}
The functions $\F_{l/r}(U_l,U_r,\Delta z)\in \r^2$ are the numerical fluxes, see \cite{bouchut_book}.\\

Indeed the method used in~\cite{PS} in order to solve~\eqref{eq:SV} can be viewed as solving
\begin{equation}
	\partial_t f+\xi\partial_x f-g(\partial_x z)\partial_\xi f=0
	\label{eq:kintrans}
\end{equation}
for the unknown $f(t,x,\xi)$, over the time interval $(t^n,t^{n+1})$,
with initial data
\begin{equation}
	f(t^n,x,\xi)=M(U^n(x),\xi).
	\label{eq:initM}
\end{equation}
Defining the update as
\begin{equation}
	U^{n+1}_i=\frac{1}{\Delta x_i}\int_{x_{i-1/2}}^{x_{i+1/2}}\int_\r \kxi f(t^{n+1-},x,\xi)\,dxd\xi,
	\label{eq:updatekin}
\end{equation}
and
\begin{equation}
	f^{n+1-}_i(\xi)=\frac{1}{\Delta x_i}\int_{x_{i-1/2}}^{x_{i+1/2}}f(t^{n+1-},x,\xi)\,dx,
	\label{eq:f_in+1-}
\end{equation} 
the formula \eqref{eq:updatekin} can then be written
\begin{equation}
	U^{n+1}_i=\int_\r \kxi f^{n+1-}_i(\xi)\,d\xi.
	\label{eq:Un+1-i}
\end{equation}
This formula can in fact be written under the form \eqref{eq:upU0}, \eqref{eq:flux_def}
for some numerical fluxes $\F_{l/r}$ computed in \cite{PS}, involving nonexplicit integrals.

A main idea in this paper is to use simplified formulas, and it will be done by defining
a suitable approximation of $f^{n+1-}_i(\xi)$.
We shall often denote $U_i$ instead of $U_i^n$, whenever there is no ambiguity.

\section{Kinetic entropy inequality without topography}
\label{sec:kin_without_topo}

In this section we consider the problem~\eqref{eq:SV} without topography,
and the unmodified kinetic scheme \eqref{eq:kintrans}, \eqref{eq:initM}, \eqref{eq:f_in+1-}, \eqref{eq:Un+1-i}.
This problem is classical, and we recall here how the entropy inequality is analyzed in this case,
in the fully discrete and semi-discrete cases.

\subsection{Fully discrete scheme}\label{full-discrete-notopo}

Without topography, the kinetic scheme is an entropy satisfying {\sl flux vector splitting} scheme \cite{FVS}.
The update \eqref{eq:f_in+1-} of the solution of~\eqref{eq:kintrans},\eqref{eq:initM} simplifies to the discrete kinetic scheme
\begin{equation}
	f_i^{n+1-}=M_i-\sigma_i\xi\Bigl(\1_{\xi>0}M_i+\1_{\xi<0}M_{i+1}
	-\1_{\xi<0}M_i-\1_{\xi>0}M_{i-1}\Bigr),
	\label{eq:kinsc}
\end{equation}
with $\sigma_i=\Delta t^n/\Delta x_i$ and with short notation (we omit the variable $\xi$).
One can write it
\begin{equation}
	f_i^{n+1-}=\left\{\begin{array}{ll}
	\dsp (1+\sigma_i\xi)M_i-\sigma_i\xi M_{i+1}&\mbox{ if }\xi<0,\\
	\dsp(1-\sigma_i\xi)M_i+\sigma_i\xi M_{i-1}&\mbox{ if }\xi>0.
\end{array}
\right.
	\label{eq:kinsconvex}
\end{equation}
Then under the CFL condition that
\begin{equation}
	\sigma_i|\xi|\leq 1 \mbox{ in the supports of }M_i, M_{i-1}, M_{i+1},
	\label{eq:CFL}
\end{equation}
$f_i^{n+1-}$ is a convex combination of $M_i$ and $M_{i+1}$ if $\xi<0$,
of $M_i$ and $M_{i-1}$ if  $\xi>0$. Thus $f_i^{n+1-}\geq 0$, and recalling the kinetic entropy $H_0(f,\xi)$ from \eqref{eq:kinH0},
we have
\begin{equation}
	H_0(f_i^{n+1-},\xi)\leq\left\{\begin{array}{ll}
	\dsp (1+\sigma_i\xi)H_0(M_i,\xi)-\sigma_i\xi H_0(M_{i+1},\xi)&\mbox{ if }\xi<0,\\
	\dsp(1-\sigma_i\xi)H_0(M_i,\xi)+\sigma_i\xi H_0(M_{i-1},\xi)&\mbox{ if }\xi>0.
\end{array}
\right.
	\label{eq:kinsconvexentro}
\end{equation}
This can be also written as
\begin{equation}\begin{array}{l}
	\dsp H_0(f_i^{n+1-},\xi)\leq H_0(M_i,\xi)-\sigma_i\xi\Bigl(\1_{\xi>0}H_0(M_i,\xi)+\1_{\xi<0}H_0(M_{i+1},\xi)\\
	\dsp\hphantom{H_0(M_i^{n+1-},\xi)\leq}
	-\1_{\xi<0}H_0(M_i,\xi)-\1_{\xi>0}H_0(M_{i-1},\xi)\Bigr),
	\end{array}
	\label{eq:kinentr}
\end{equation}
which can be interpreted as a conservative kinetic entropy inequality.
Note that with \eqref{eq:Un+1-i} and \eqref{eq:entropmini},
\begin{equation}
	\eta(U^{n+1}_i)\leq \int_\r\,H_0(f_i^{n+1-}(\xi),\xi)d\xi,
	\label{eq:entrnotopo}
\end{equation}
which by integration of \eqref{eq:kinentr} yields the macroscopic entropy inequality.

The scheme~\eqref{eq:kinsc} and the definition~\eqref{eq:Un+1-i} allow to complete the definition of the
macroscopic scheme~\eqref{eq:upU0}, \eqref{eq:flux_def} with the numerical flux $\F_l=\F_r\equiv\F$
given by the {\sl flux vector splitting} formula \cite{FVS}
\begin{equation}
	\F(U_l,U_r)=\int_{\xi>0}\xi \kxi M(U_l,\xi)\,d\xi+\int_{\xi<0}\xi\kxi M(U_r,\xi)\,d\xi,
	\label{eq:kinupwind}
\end{equation}
where $M$ is defined in \eqref{eq:kinmaxw}.

\subsection{Semi-discrete scheme}\label{semi-no-topo}

Assuming that the timestep is very small (i.e. $\sigma_i$ very small),
we have the linearized approximation of the entropy variation from \eqref{eq:kinsc}
\begin{equation}\begin{array}{l}
	\dsp H_0(f_i^{n+1-},\xi)\simeq H_0(M_i,\xi)-\sigma_i\xi \partial_f H_0(M_i,\xi)\Bigl(\1_{\xi>0}M_i+\1_{\xi<0}M_{i+1}\\
	\dsp\hphantom{H_0(M_i^{n+1-},\xi)\leq}
	-\1_{\xi<0}M_i-\1_{\xi>0}M_{i-1}\Bigr).
	\end{array}
	\label{eq:kinentrsemi}
\end{equation}
This linearization with respect to $\Delta t^n$
(or equivalently with respect to $\sigma_i=\Delta t^n/\Delta x_i$) represents indeed
the entropy in the semi-discrete limit $\Delta t^n\to 0$ (divide \eqref{eq:kinentrsemi}
by $\Delta t^n$ and let formally $\Delta t^n\to 0$).
The entropy inequality attached to this linearization can be estimated as follows.

\begin{lemma} The linearized term from \eqref{eq:kinentrsemi} is dominated by
the conservative difference from \eqref{eq:kinentr},
\begin{equation}\begin{array}{l}
	\dsp -\sigma_i\xi \partial_f H_0(M_i,\xi)\Bigl(\1_{\xi>0}M_i+\1_{\xi<0}M_{i+1}
	-\1_{\xi<0}M_i-\1_{\xi>0}M_{i-1}\Bigr)\\
	\dsp\leq -\sigma_i\xi\Bigl(\1_{\xi>0}H_0(M_i,\xi)+\1_{\xi<0}H_0(M_{i+1},\xi)\\
	\dsp\hphantom{H_0(M_i^{n+1-},\xi)\leq}
	-\1_{\xi<0}H_0(M_i,\xi)-\1_{\xi>0}H_0(M_{i-1},\xi)\Bigr).
	\label{eq:estsemi}
	\end{array}
\end{equation}
In particular, the semi-discrete scheme is more dissipative than the fully discrete scheme.
\label{lemma:seminotopo}
\end{lemma}
\begin{proof}
It is enough to prove two inequalities,
\begin{equation}
	\xi \partial_f H_0(M_i)(\1_{\xi>0}M_i+\1_{\xi<0}M_{i+1}-M_i)
	\geq \xi(\1_{\xi>0}H_0(M_i)+\1_{\xi<0}H_0(M_{i+1})-H_0(M_i))
	\label{eq:ineqlocconsl}
\end{equation}
and
\begin{equation}
	\xi \partial_f H_0(M_i)(\1_{\xi<0}M_i+\1_{\xi>0}M_{i-1}-M_i)
	\leq \xi(\1_{\xi<0}H_0(M_i)+\1_{\xi>0}H_0(M_{i-1})-H_0(M_i)).
	\label{eq:ineqlocconsr}
\end{equation}
We observe that \eqref{eq:ineqlocconsl} is trivial for $\xi>0$, and \eqref{eq:ineqlocconsr} is trivial for $\xi<0$.
The two conditions can therefore be written
\begin{equation}\begin{array}{l}
	\dsp \partial_f H_0(M_i)(M_{i+1}-M_i)\leq H_0(M_{i+1})-H_0(M_i)\quad\mbox{for }\xi<0,\\
	\dsp \partial_f H_0(M_i)(M_{i-1}-M_i)\leq H_0(M_{i-1})-H_0(M_i)\quad\mbox{for }\xi>0.
	\label{eq:ineqconvloc}
	\end{array}
\end{equation}
These last inequalities follow from the convexity of $H_0$.
\end{proof}

\section{Kinetic interpretation of the hydrostatic reconstruction scheme}
\label{sec:kin_HR}

The hydrostatic reconstruction scheme (HR scheme for short) for the Saint-Venant system
\eqref{eq:SV}, has been introduced in \cite{ABBKP}, and can be written as follows,
\begin{equation}
	U^{n+1}_i=U_i-\sigma_i(F_{i+1/2-}-F_{i-1/2+}),
	\label{eq:upU}
\end{equation}
where $\sigma_i=\Delta t^n/\Delta x_i$,
\begin{equation}\begin{array}{l}
	\dsp F_{i+1/2-}=\F(U_{i+1/2-},U_{i+1/2+})+\begin{pmatrix}0\\ g\frac{h_i^2}{2}-\frac{gh_{i+1/2-}^2}{2}\end{pmatrix},\\
	\dsp F_{i+1/2+}=\F(U_{i+1/2-},U_{i+1/2+})+\begin{pmatrix}0\\ g\frac{h_{i+1}^2}{2}-\frac{gh_{i+1/2+}^2}{2}\end{pmatrix},
	\label{eq:HRflux}
	\end{array}
\end{equation}
$\F$ is a numerical flux for the system without topography, and the reconstructed states
\begin{equation}
	U_{i+1/2-}=(h_{i+1/2-},h_{i+1/2-}u_i),\qquad
	U_{i+1/2+}=(h_{i+1/2+},h_{i+1/2+}u_{i+1}),
	\label{eq:lrstates}
\end{equation}
are defined by
\begin{equation}
	h_{i+1/2-}=(h_i+z_i-z_{i+1/2})_+,\qquad
	h_{i+1/2+}=(h_{i+1}+z_{i+1}-z_{i+1/2})_+,
	\label{eq:hlr}
\end{equation}
and
\begin{equation}
	z_{i+1/2}=\max(z_i,z_{i+1}).
	\label{eq:zstar}
\end{equation}

We would like here to propose a kinetic interpretation of the HR scheme, which means to
interpret the above numerical fluxes as averages with respect to the kinetic variable
of a scheme written on a kinetic function $f$.
More precisely, we would like to approximate the solution to \eqref{eq:kintrans} by a kinetic scheme
such that the associated macroscopic scheme is exactly \eqref{eq:upU}-\eqref{eq:HRflux}
with homogeneous numerical flux $\F$ given by~\eqref{eq:kinupwind}.
We denote $M_i=M(U_i,\xi)$, $M_{i+1/2\pm}=M(U_{i+1/2\pm},\xi)$, $f_i^{n+1-}=f_i^{n+1-}(\xi)$,
and we consider the scheme
\begin{equation}\begin{array}{l}
	\dsp f_i^{n+1-}=M_i-\sigma_i\biggl(\xi\1_{\xi<0}M_{i+1/2+}+\xi\1_{\xi>0}M_{i+1/2-}+\delta M_{i+1/2-}\\
	\dsp\mkern 160mu -\xi\1_{\xi>0}M_{i-1/2-}-\xi\1_{\xi<0}M_{i-1/2+}-\delta M_{i-1/2+}\biggr).
	\label{eq:kinups}
	\end{array}
\end{equation}
In this formula, $\delta M_{i+1/2\pm}$ depend on $\xi$, $U_i$, $U_{i+1}$, $\Delta z_{i+1/2}=z_{i+1}-z_i$,
and are assumed to satisfy the moment relations
\begin{equation}
	\int_\r\delta M_{i+1/2-}\,d\xi=0,\quad
	\int_\r\xi\,\delta M_{i+1/2-}\,d\xi=g\frac{h_i^2}{2}-g\frac{h_{i+1/2-}^2}{2},
	\label{eq:intdeltaM-}
\end{equation}
\begin{equation}
	\int_\r\delta M_{i-1/2+}\,d\xi=0,\quad
	\int_\r\xi\,\delta M_{i-1/2+}\,d\xi=g\frac{h_i^2}{2}-g\frac{h_{i-1/2+}^2}{2}.
	\label{eq:intdeltaM+}
\end{equation}
Using again~\eqref{eq:Un+1-i},
the integration of \eqref{eq:kinups} multiplied by $\kxi$
with respect to $\xi$ then gives the HR scheme \eqref{eq:upU}-\eqref{eq:HRflux}
with \eqref{eq:lrstates}-\eqref{eq:zstar}, \eqref{eq:kinupwind}.
Thus as announced, \eqref{eq:kinups} is a kinetic interpretation of the HR scheme.
The remainder of this section is devoted to its analysis.

\subsection{Analysis of the semi-discrete scheme}\label{semitopo}

Assuming that the timestep is very small (i.e. $\sigma_i$ very small),
we have the linearized approximation of the entropy variation from \eqref{eq:kinups},
\begin{equation}\begin{array}{l}
	\dsp H(f_i^{n+1-},z_i)\simeq H(M_i,z_i)-\sigma_i\partial_f H(M_i,z_i)\biggl(\xi\1_{\xi<0}M_{i+1/2+}+\xi\1_{\xi>0}M_{i+1/2-}\\
	\dsp\mkern 60mu +\delta M_{i+1/2-}-\xi\1_{\xi>0}M_{i-1/2-}-\xi\1_{\xi<0}M_{i-1/2+}-\delta M_{i-1/2+}\biggr),
	\label{eq:kinupsH}
	\end{array}
\end{equation}
where the kinetic entropy $H(f,\xi,z)$ is defined in \eqref{eq:kinH}.
As in Subsection \ref{semi-no-topo}, this linearization with respect to $\sigma_i=\Delta t^n/\Delta x_i$
represents indeed the entropy in the semi-discrete limit $\Delta t^n\to 0$.
Its dissipation can be estimated as follows.
\begin{proposition}\label{prop semi-discrete}
We assume that the extra variations $\delta M_{i+1/2\pm}$
satisfy \eqref{eq:intdeltaM-}, \eqref{eq:intdeltaM+}, and also
\begin{equation}
	M(U_i,\xi)=0\ \Rightarrow \delta M_{i+1/2-}(\xi)=0\mbox{ and }\delta M_{i-1/2+}(\xi)=0.
	\label{eq:support_deltai+1/2}
\end{equation}
Then the linearized term from \eqref{eq:kinupsH} is dominated by
a quasi-conservative difference,
\begin{equation}\begin{array}{l}
	\dsp \partial_f H(M_i,z_i)\biggl(\xi\1_{\xi<0}M_{i+1/2+}+\xi\1_{\xi>0}M_{i+1/2-}\\
	\dsp\mkern 60mu +\delta M_{i+1/2-}-\xi\1_{\xi>0}M_{i-1/2-}-\xi\1_{\xi<0}M_{i-1/2+}-\delta M_{i-1/2+}\biggr)\\
	\dsp\geq \widetilde H_{i+1/2-}-\widetilde H_{i-1/2+},
	\end{array}
	\label{eq:hquasidiff}
\end{equation}
where
\begin{equation}\begin{array}{l}
	\dsp\widetilde H_{i+1/2-}=\xi\1_{\xi<0}H(M_{i+1/2+},z_{i+1/2})+\xi\1_{\xi>0}H(M_{i+1/2-},z_{i+1/2})\\
	\dsp\hphantom{\widetilde H_{i+1/2-}=}+\xi H(M_i,z_i)-\xi H(M_{i+1/2-},z_{i+1/2})\vphantom{\Bigl|}\\
	\dsp\hphantom{\widetilde H_{i+1/2-}=}+\Bigl(\eta'(U_i)\kxi+gz_i\Bigr)
	\bigl(\xi M_{i+1/2-}-\xi M_i+\delta M_{i+1/2-}\bigr),
	\end{array}
	\label{eq:tildH-}
\end{equation}
\begin{equation}\begin{array}{l}
	\dsp\widetilde H_{i-1/2+}=\xi\1_{\xi<0}H(M_{i-1/2+},z_{i-1/2})+\xi\1_{\xi>0}H(M_{i-1/2-},z_{i-1/2})\\
	\dsp\hphantom{\widetilde H_{i+1/2-}=}+\xi H(M_i,z_i)-\xi H(M_{i-1/2+},z_{i-1/2})\vphantom{\Bigl|}\\
	\dsp\hphantom{\widetilde H_{i+1/2-}=}+\Bigl(\eta'(U_i)\kxi+gz_i\Bigr)
	\bigl(\xi M_{i-1/2+}-\xi M_i+\delta M_{i-1/2+}\bigr).
	\end{array}
	\label{eq:tildH+}
\end{equation}
Moreover, the integral with respect to $\xi$ of the last two lines of \eqref{eq:tildH-} (respectively 
of \eqref{eq:tildH+}) vanishes. In particular,
\begin{equation}
	\int_\r\bigl(\widetilde H_{i+1/2-}-\widetilde H_{i-1/2+}\bigr)\,d\xi=\widetilde G_{i+1/2}-\widetilde G_{i-1/2},
	\label{eq:diffkinH}
\end{equation}
with
\begin{equation}
	\widetilde G_{i+1/2}=\int_{\xi<0}\xi H(M_{i+1/2+},z_{i+1/2})\,d\xi+\int_{\xi>0}\xi H(M_{i+1/2-},z_{i+1/2})\,d\xi.
	\label{eq:kingtilde}
\end{equation}
\end{proposition}
\begin{proof}
The value of the integral with respect to $\xi$ of the two last lines of \eqref{eq:tildH-} is
\begin{equation}\begin{array}{l}
	\dsp\hphantom{=} \bigl(h_i\frac{u_i^2}{2}+gh_i^2+gh_iz_i\bigr)u_i-\bigl(h_{i+1/2-}\frac{u_i^2}{2}+gh_{i+1/2-}^2+gh_{i+1/2-}z_{i+1/2}\bigr)u_i\\
	\dsp\hphantom{=}+(gh_i+gz_i-u_i^2/2)u_i(h_{i+1/2-}-h_i)+u_i^3(h_{i+1/2-}-h_i)\\
	\dsp =u_igh_{i+1/2-}(-h_{i+1/2-}-z_{i+1/2}+z_i+h_i)\\
	\dsp =0,
	\label{eq:compintlines}
	\end{array}
\end{equation}
because of the definition of $h_{i+1/2-}$ in \eqref{eq:hlr}.
The computation for \eqref{eq:tildH+}) is similar.
In order to prove \eqref{eq:hquasidiff}, it is enough to prove the two inequalities
\begin{equation}\begin{array}{l}
	\dsp \partial_f H(M_i,z_i)\biggl(\xi\1_{\xi<0}M_{i+1/2+}+\xi\1_{\xi>0}M_{i+1/2-}+\delta M_{i+1/2-}-\xi M_i\biggr)\\
	\dsp \geq \widetilde H_{i+1/2-}-\xi H(M_i,z_i),
	\end{array}
	\label{eq:ineqeleft}
\end{equation}
and
\begin{equation}\begin{array}{l}
	\dsp \partial_f H(M_i,z_i)\biggl(\xi\1_{\xi>0}M_{i-1/2-}+\xi\1_{\xi<0}M_{i-1/2+}+\delta M_{i-1/2+}-\xi M_i\biggr)\\
	\dsp \leq\widetilde H_{i-1/2+}-\xi H(M_i,z_i).
	\end{array}
	\label{eq:ineqeright}
\end{equation}
We note that the definitions of $h_{i+1/2\pm}$ in \eqref{eq:hlr}-\eqref{eq:zstar} ensure
that $h_{i+1/2-}\leq h_i$, and $h_{i+1/2+}\leq h_{i+1}$.
Therefore,  because of \eqref{eq:kinmaxw} one has
\begin{equation}
	0\leq M_{i+1/2-}\leq M_i,\quad 0\leq M_{i+1/2+}\leq M_{i+1},
	\label{eq:domM_i+1/2}
\end{equation}
thus
\begin{equation}
	M(U_i,\xi)=0\ \Rightarrow M(U_{i+1/2-},\xi)=0\mbox{ and }M(U_{i-1/2+},\xi)=0.
	\label{eq:supporti+1/2}
\end{equation}
Taking into account \eqref{eq:support_deltai+1/2}, with \eqref{eq:idH'} we get
\begin{equation}\begin{array}{l}
	\dsp\hphantom{=} \Bigl(\eta'(U_i)\kxi+gz_i\Bigr)\bigl(\xi M_{i+1/2-}-\xi M_i+\delta M_{i+1/2-}\bigr)\\
	\dsp=\partial_f H(M_i,z_i)\bigl(\xi M_{i+1/2-}-\xi M_i+\delta M_{i+1/2-}\bigr),
	\label{eq:iidleft}
	\end{array}
\end{equation}
and
\begin{equation}\begin{array}{l}
	\dsp\hphantom{=}  \Bigl(\eta'(U_i)\kxi+gz_i\Bigr)\bigl(\xi M_{i-1/2+}-\xi M_i+\delta M_{i-1/2+}\bigr)\\
	\dsp =\partial_f H(M_i,z_i)\bigl(\xi M_{i-1/2+}-\xi M_i+\delta M_{i-1/2+}\bigr).
	\label{eq:eq:idright}
	\end{array}
\end{equation}
Therefore, the inequalities \eqref{eq:ineqeleft}-\eqref{eq:ineqeright} simplify to
\begin{equation}\begin{array}{l}
	\dsp \partial_f H(M_i,z_i)\biggl(\xi\1_{\xi<0}M_{i+1/2+}+\xi\1_{\xi>0}M_{i+1/2-}-\xi M_{i+1/2-}\biggr)\\
	\dsp \geq \xi\1_{\xi<0}H(M_{i+1/2+},z_{i+1/2})+\xi\1_{\xi>0}H(M_{i+1/2-},z_{i+1/2})
	-\xi H(M_{i+1/2-},z_{i+1/2})\vphantom{\Bigl|},
	\end{array}
	\label{eq:ineqeleftb}
\end{equation}
\begin{equation}\begin{array}{l}
	\dsp \partial_f H(M_i,z_i)\biggl(\xi\1_{\xi>0}M_{i-1/2-}+\xi\1_{\xi<0}M_{i-1/2+}-\xi M_{i-1/2+}\biggr)\\
	\dsp \leq\xi\1_{\xi<0}H(M_{i-1/2+},z_{i-1/2})+\xi\1_{\xi>0}H(M_{i-1/2-},z_{i-1/2})
	-\xi H(M_{i-1/2+},z_{i-1/2})\vphantom{\Bigl|}.
	\end{array}
	\label{eq:ineqerightb}
\end{equation}
The first inequality \eqref{eq:ineqeleftb} is trivial for $\xi>0$,
and the second inequality \eqref{eq:ineqerightb} is trivial for $\xi<0$.
Therefore it is enough to satisfy the two inequalities
\begin{equation}
	\partial_f H(M_i,z_i)\Bigl(M_{i+1/2+}-M_{i+1/2-}\Bigr)
	\leq H(M_{i+1/2+},z_{i+1/2})-H(M_{i+1/2-},z_{i+1/2}),
	\label{eq:ineqsubconvleft}
\end{equation}
\begin{equation}
	\partial_f H(M_i,z_i)\Bigl(M_{i-1/2-}-M_{i-1/2+}\Bigr)
	\leq H(M_{i-1/2-},z_{i-1/2})-H(M_{i-1/2+},z_{i-1/2}).
	\label{eq:ineqsubconveright}
\end{equation}
But as in Subsection \ref{semi-no-topo}, we have according to the convexity of $H$ with
respect to $f$,
\begin{equation}\begin{array}{l}
	\dsp H(M_{i+1/2+},z_{i+1/2})\geq H(M_{i+1/2-},z_{i+1/2})\\
	\dsp\mkern 200mu
	+\partial_f H(M_{i+1/2-},z_{i+1/2})(M_{i+1/2+}-M_{i+1/2-}),
	\label{eq:convleft}
	\end{array}
\end{equation}
\begin{equation}\begin{array}{l}
	\dsp H(M_{i-1/2-},z_{i-1/2})\geq H(M_{i-1/2+},z_{i-1/2})\\
	\dsp\mkern 200mu
	+\partial_f H(M_{i-1/2+},z_{i-1/2})(M_{i-1/2-}-M_{i-1/2+}).
	\label{eq:convright}
	\end{array}
\end{equation}
In order to prove \eqref{eq:ineqsubconvleft}, we observe that if $M_i(\xi)=0$ then
$M_{i+1/2-}(\xi)=0$ also, thus $\partial_f H(M_{i+1/2-},z_{i+1/2})-\partial_f H(M_i,z_i)=g(z_{i+1/2}-z_i)\geq 0$
because of \eqref{eq:zstar},  and
the inequality \eqref{eq:ineqsubconvleft} follows from \eqref{eq:convleft}.
Next, if $M_i(\xi)>0$, one has
\begin{equation}\begin{array}{l}
	\dsp \hphantom{=}\partial_f H(M_i,z_i)(M_{i+1/2+}-M_{i+1/2-})\\
	\dsp =\bigl(\eta'(U_i)\kxi+gz_i\bigr)(M_{i+1/2+}-M_{i+1/2-}),
	\label{eq:idprimeleft}
	\end{array}
\end{equation}
and as in \eqref{eq:convineqH_0}
\begin{equation}\begin{array}{l}
	\dsp\hphantom{\geq} \partial_f H(M_{i+1/2-},z_{i+1/2})(M_{i+1/2+}-M_{i+1/2-})\\
	\dsp\geq \bigl(\eta'(U_{i+1/2-})\kxi+gz_{i+1/2}\bigr)(M_{i+1/2+}-M_{i+1/2-}).
	\label{eq:leftpos}
	\end{array}
\end{equation}
Taking the difference between \eqref{eq:leftpos} and \eqref{eq:idprimeleft}, we obtain
\begin{equation}\begin{array}{l}
	\dsp \hphantom{\geq}\partial_f H(M_{i+1/2-},z_{i+1/2})(M_{i+1/2+}-M_{i+1/2-})-\partial_f H(M_i,z_i)(M_{i+1/2+}-M_{i+1/2-})\\
	\dsp\geq\bigl(gh_{i+1/2-}-gh_i+gz_{i+1/2}-gz_i\bigr)(M_{i+1/2+}-M_{i+1/2-})\geq 0,
	\end{array}
	\label{eq:keyleft}
\end{equation}
because of the definition \eqref{eq:hlr} of $h_{i+1/2-}$.
Therefore we conclude that in any case ($M_i(\xi)$ being zero or not),
one has
\begin{equation}\begin{array}{l}
	\dsp \partial_f H(M_i,z_i)(M_{i+1/2+}-M_{i+1/2-})
	-H(M_{i+1/2+},z_{i+1/2})+H(M_{i+1/2-},z_{i+1/2})\\
	\leq H(M_{i+1/2-},z_{i+1/2})-H(M_{i+1/2+},z_{i+1/2})\\
	\dsp\mkern 100mu
	+\partial_f H(M_{i+1/2-},z_{i+1/2})(M_{i+1/2+}-M_{i+1/2-})\\
	\leq 0
	\label{eq:estdissleft}
	\end{array}
\end{equation}
because of \eqref{eq:convleft}, and this proves \eqref{eq:ineqsubconvleft}.
Similarly one gets
\begin{equation}\begin{array}{l}
	\dsp \partial_f H(M_i,z_i)(M_{i-1/2-}-M_{i-1/2+})
	-H(M_{i-1/2-},z_{i-1/2})+H(M_{i-1/2+},z_{i-1/2})\\
	\leq H(M_{i-1/2+},z_{i-1/2})-H(M_{i-1/2-},z_{i-1/2})\\
	\dsp\mkern 100mu
	+\partial_f H(M_{i-1/2+},z_{i-1/2})(M_{i-1/2-}-M_{i-1/2+})\\
	\leq 0,
	\label{eq:estdissright}
	\end{array}
\end{equation}
proving \eqref{eq:ineqsubconveright}. This concludes the proof,
and we observe that we have indeed a dissipation estimate slightly stronger than \eqref{eq:hquasidiff},
\begin{equation}\begin{array}{l}
	\dsp \partial_f H(M_i,z_i)\biggl(\xi\1_{\xi<0}M_{i+1/2+}+\xi\1_{\xi>0}M_{i+1/2-}\\
	\dsp\mkern 60mu +\delta M_{i+1/2-}-\xi\1_{\xi>0}M_{i-1/2-}-\xi\1_{\xi<0}M_{i-1/2+}-\delta M_{i-1/2+}\biggr)\\
	\dsp\geq \widetilde H_{i+1/2-}-\widetilde H_{i-1/2+}\\
	\dsp -\xi\1_{\xi<0}\Bigl(H(M_{i+1/2+},z_{i+1/2})-H(M_{i+1/2-},z_{i+1/2})\\
	\dsp\mkern 100mu
	-\partial_f H(M_{i+1/2-},z_{i+1/2})(M_{i+1/2+}-M_{i+1/2-})\Bigr)\\
	\dsp +\xi\1_{\xi>0}\Bigl(H(M_{i-1/2-},z_{i-1/2})-H(M_{i-1/2+},z_{i-1/2})\\
	\dsp\mkern 100mu
	-\partial_f H(M_{i-1/2+},z_{i-1/2})(M_{i-1/2-}-M_{i-1/2+})\Bigr).
	\end{array}
	\label{eq:hquasidiffdiss}
\end{equation}
\end{proof}
\begin{remark}
The numerical entropy flux \eqref{eq:kingtilde} can be written
\begin{equation}
	\widetilde G_{i+1/2}={\cal G}(U_{i+1/2-},U_{i+1/2+})+gz_{i+1/2}\F^0(U_{i+1/2-},U_{i+1/2+}),
	\label{eq:numentrflux}
\end{equation}
where ${\cal G}$ is the numerical entropy flux of the scheme without topography, and $\F^0$ is
the first component of $\F$. This formula is in accordance of the analysis of the semi-discrete
entropy inequality in \cite{ABBKP}.
\end{remark}
\begin{remark}
At the kinetic level, the entropy inequality \eqref{eq:hquasidiff} is not in conservative form.
The entropy inequality becomes conservative only when taking the integral with respect to $\xi$,
as is seen on \eqref{eq:diffkinH}. This is also the case in \cite{PS}.
Indeed we have written the macroscopic conservative entropy inequality as an integral with respect to $\xi$
of the sum of a nonpositive term (the one in \eqref{eq:hquasidiff}),
a kinetic conservative term (the difference of the first lines of \eqref{eq:tildH-} and  \eqref{eq:tildH+}),
and a term with vanishing integral (difference of the two last lines of \eqref{eq:tildH-} and  \eqref{eq:tildH+}).
However, such a decomposition is not unique.
\end{remark}

\subsection{Analysis of the fully discrete scheme}

We still consider the scheme \eqref{eq:kinups}, and we make the choice
\begin{equation}\begin{array}{l}
	\dsp \delta M_{i+1/2-}=(\xi-u_i)(M_i-M_{i+1/2-}),\\
	\dsp \delta M_{i-1/2+}=(\xi-u_i)(M_i-M_{i-1/2+}),
	\label{eq:defdeltaM}
	\end{array}
\end{equation}
that satisfies the assumptions \eqref{eq:intdeltaM-}, \eqref{eq:intdeltaM+}
and \eqref{eq:support_deltai+1/2}.
The scheme \eqref{eq:kinups} is therefore a kinetic interpretation of the HR scheme
\eqref{eq:upU}-\eqref{eq:zstar}.
\begin{lemma}\label{kwb}
The scheme \eqref{eq:kinups} with the choice \eqref{eq:defdeltaM} is ``kinetic well-balanced''
for steady states at rest, and consistent with \eqref{eq:kintrans}.
\end{lemma}
\begin{proof}
The expression kinetic well-balanced means that we do not
only prove that
\begin{equation}
	\int_\r\kxi f^{n+1-}_i\, d\xi = \int_\r \kxi M_i\, d\xi,
	\label{eq:wbmacro}
\end{equation}
at rest, but the stronger property
\begin{equation}
f_{i}^{n+1-}(\xi) = M_{i}(\xi),\quad \forall\xi\in\r,
\label{eq:rest_kin}
\end{equation}
when $u_i=0$ and $h_i + z_{i} =h_{i+1} + z_{i+1}$ for all $i$.
Indeed in this situation one has $U_{i+1/2-}=U_{i+1/2+}$ for all $i$,
thus the first three terms between parentheses in \eqref{eq:kinups} give $\xi M_i$,
and the last three terms give $-\xi M_i$, leading to \eqref{eq:rest_kin}.

The consistency of the HR scheme has been proved in \cite{ABBKP}, but here the statement
is the consistency of the kinetic update \eqref{eq:kinups} with the kinetic equation \eqref{eq:kintrans}.
We proceed as follows. Using~\eqref{eq:initM} and \eqref{eq:kinmaxw},
the topography source term in~\eqref{eq:kintrans} reads
\begin{equation}
-g (\partial_x z) \partial_\xi M = g (\partial_x z) \frac{\xi
-u}{2gh - (\xi - u)^2}M.
\label{eq:vlasov}
\end{equation}
This formula is valid for $2gh - (\xi - u)^2\not=0$, i.e. when $\xi \neq u \pm \sqrt{2gh}$ or in $L^1(\xi\in\r)$.
Assuming that $h_i>0$ (otherwise the consistency is obvious),
one has that $h_{i+1/2-} = h_i + z_i - z_{i+1/2}$ for $z_{i+1}-z_i$ small enough,
and an asymptotic expansion of $M_{i+1/2-}$ gives
\begin{equation}
M_{i+1/2-} = M_i + (z_i - z_{i+1/2})(\partial_{h_i} M_i)_{|u_i} + o (z_{i+1} -z_{i}),
\label{eq:expand}
\end{equation}
with
\begin{equation}
	(\partial_{h_i} M_i)_{|u_i}= g \frac{M_i}{2gh_i - (\xi - u_i)^2}.
	\label{eq:dhM}
\end{equation}
Thus
\begin{equation}
	\frac{\delta M_{i+1/2-}}{\Delta x_i}
	=  g \frac{z_{i+1/2} -z_i}{\Delta x_i}\frac{\xi - u_i}{2gh_i - (\xi - u_i)^2} M_i+ o(1).
	\label{eq:expandbis}
\end{equation}
Similarly, one has
\begin{equation}
	\frac{\delta M_{i-1/2+}}{\Delta x_i}
	=  g \frac{z_{i-1/2} -z_i}{\Delta x_i}\frac{\xi - u_i}{2gh_i - (\xi - u_i)^2} M_i+ o(1).
	\label{eq:expandbis+}
\end{equation}
With the usual shift of index $i$ due to the distribution of the source to interfaces,
the difference \eqref{eq:expandbis} minus \eqref{eq:expandbis+}
appears as a discrete version of~\eqref{eq:vlasov}.
The other four terms in parentheses in \eqref{eq:kinups} are conservative,
and are classically consistent with $\xi\partial_x f$ in \eqref{eq:kintrans}.
\end{proof}
\begin{remark} The scheme \eqref{eq:kinups} can be viewed as a consistent well-balanced
scheme for \eqref{eq:kintrans}, except that the notion of consistency is true here only for
Maxwellian initial data. On the contrary, the exact solution used in \cite{PS} is consistent for
initial data of arbitrary shape. The role of the special form of the Maxwellian \eqref{eq:kinmaxw}
is seen here by the fact that for initial data $U_i$ at rest, one has that $M(U_i,\xi)$ is a
steady state of \eqref{eq:kintrans} (this results from \eqref{eq:vlasov} and \eqref{eq:dhM}).
\end{remark}

When writing the entropy inequality for the fully discrete scheme, the difficulty is to estimate
the positive part of the entropy dissipation by something that tends to zero
when $\Delta x_i$ tends to zero, at constant Courant number $\sigma_i$,
and assuming only that $\Delta z/\Delta x$ is bounded (Lipschitz topography), but not
that $\Delta U/\Delta x$ is bounded (the solution can have discontinuities).
Here $\Delta z$ stands for a quantity like $z_{i+1}-z_i$, and $\Delta U$ stands
for a quantity like $U_{i+1}-U_i$.

The principle of proof of such entropy inequality is that we use the dissipation of the
semi-discrete scheme proved in Proposition \ref{prop semi-discrete}, under the strong form \eqref{eq:hquasidiffdiss}.
This inequality involves the terms linear in $\sigma_i$. Under a CFL condition,
the higher order terms (quadratic in $\sigma_i$ or higher) are either treated
as errors if they are of the order of $\Delta z^2$ or $\Delta z\Delta U$,
or must be dominated by the dissipation if they are of the order of $\Delta U^2$.
Note that the dissipation in \eqref{eq:hquasidiffdiss}, i.e. the two last expressions in factor of
$\1_{\xi<0}$ and $\1_{\xi>0}$ respectively, are of the order of $(M_{i+1/2+}-M_{i+1/2-})^2$
and $(M_{i-1/2+}-M_{i-1/2-})^2$ respectively, and thus neglecting the terms in $\Delta z$,
they control $(M_{i+1}-M_{i})^2$ and $(M_{i}-M_{i-1})^2$ respectively.
However, the Maxwellian \eqref{eq:kinmaxw} is not Lipschitz continuous with respect to $U$,
thus a sharp analysis has to be performed in order to use the dissipation.\\

We consider a velocity $v_m\geq 0$ such that for all $i$,
\begin{equation}
	M(U_i,\xi)>0\Rightarrow |\xi|\leq  v_m.
	\label{eq:vmax}
\end{equation}
This means equivalently that $|u_i|+\sqrt{2gh_i}\leq v_m$.
We consider a CFL condition strictly less than one,
\begin{equation}
	\sigma_iv_m\leq \beta<1 \quad \mbox{ for all }i,
	\label{eq:CFLfull}
\end{equation}
where $\sigma_i=\Delta t^n/\Delta x_i$, and $\beta$ is a given constant.
\begin{theorem}\label{thentrfull_JSM}
Under the CFL condition~\eqref{eq:CFLfull}, the scheme \eqref{eq:kinups} with the choice \eqref{eq:defdeltaM} verifies the following properties.

\noindent (i) The kinetic function remains nonnegative $f^{n+1-}_i\geq 0$.

\noindent (ii) One has the kinetic entropy inequality
\begin{equation}\begin{array}{l}
	\dsp\hphantom{\leq}\ H(f_i^{n+1-},z_i)\\
	\dsp\leq H(M_i,z_i)-\sigma_i\Bigl(\widetilde{H}_{i+1/2-} -
   \widetilde{H}_{i-1/2+}\Bigr)\\
	\dsp\hphantom{\leq}-\nu_\beta\,\sigma_i|\xi|\frac{g^2\pi^2}{6}
	\biggl(\1_{\xi<0}\,(M_{i+1/2+}+M_{i+1/2-})(M_{i+1/2+} -  M_{i+1/2-})^2\\
	\dsp\hphantom{\leq}+\1_{\xi>0}\,(M_{i-1/2-}+M_{i-1/2+})(M_{i-1/2+} -  M_{i-1/2-})^2\biggr)\\
	\dsp\hphantom{\leq}+C_\beta(\sigma_iv_m)^2\frac{g^2 \pi^2}{6}
	M_i\Bigl((M_i-M_{i+1/2-})^2+(M_i-M_{i-1/2+})^2\Bigr),
	\label{eq:entrfullystat}
   \end{array}
\end{equation}
where $\widetilde{H}_{i+1/2-}$, $\widetilde{H}_{i-1/2+}$ are defined
by~\eqref{eq:tildH-},\eqref{eq:tildH+}, $\nu_\beta>0$ is a dissipation constant
depending only on $\beta$, and $C_\beta\geq 0$ is a constant depending only on $\beta$.
The term proportional to $C_\beta$ is an error, while the term
proportional to $\nu_\beta$ is a dissipation that reinforces the inequality.
\end{theorem}
Theorem~\ref{thentrfull_JSM} has the following corollary.
\begin{corollary}
Under the CFL condition \eqref{eq:vmax}, \eqref{eq:CFLfull},
integrating the estimate \eqref{eq:entrfullystat} with respect to $\xi$,
using~\eqref{eq:entropmini}, \eqref{eq:Un+1-i}, \eqref{eq:diffkinH}
(neglecting the dissipation proportional to $\nu_\beta$) and Lemma~\ref{lemma lipestimate} yields that
\begin{equation}\begin{array}{l}
	\dsp \eta(U_i^{n+1})+gz_ih_i^{n+1}\leq\eta(U_i)+gz_ih_i
	-\sigma_i\Bigl(\widetilde G_{i+1/2}-\widetilde G_{i-1/2}\Bigr)\\
	\dsp\hphantom{\eta(U_i^{n+1})+gz_ih_i^{n+1}\leq}
	+C_\beta (\sigma_iv_m)^2\biggl(g(h_i-h_{i+1/2-})^2+g(h_i-h_{i-1/2+})^2 \biggr),
	\label{eq:estentrint}
	\end{array}
\end{equation}
where $\widetilde G_{i+1/2}$ is defined in \eqref{eq:kingtilde} or equivalently \eqref{eq:numentrflux},
and $C_\beta\geq 0$ depends only on $\beta$.
This is the discrete entropy inequality associated to the HR scheme \eqref{eq:upU}-\eqref{eq:zstar}
with kinetic homogeneous numerical flux \eqref{eq:kinupwind}.
With \eqref{eq:lrstates}-\eqref{eq:zstar} one has
\begin{equation}
	0\leq h_i-h_{i+1/2-}\leq|z_{i+1}-z_i|,\quad 0\leq h_i-h_{i-1/2+}\leq |z_i-z_{i-1}|.
	\label{eq:errestexpl}
\end{equation}
We conclude that the quadratic error terms proportional to $C_\beta$ in the right-hand side of \eqref{eq:estentrint}
(divide \eqref{eq:estentrint} by $\Delta t^n$ to be consistent with \eqref{eq:entropySV})
has the following key properties: it vanishes identically when $z=cst$ (no topography)
or when $\sigma_i\rightarrow 0$ (semi-discrete limit), and as soon as the topography
is Lipschitz continuous, it tends to zero strongly when the grid size tends to $0$
(consistency with the continuous entropy inequality \eqref{eq:entropySV}),
even if the solution contains shocks.
\label{corol:estim_lip_l1}
\end{corollary}
We state now a counter result saying that it is not possible to remove
the error term in~\eqref{eq:estentrint}.
It is indeed true for the HR scheme even if the homogeneous flux used is not the kinetic one.
\begin{proposition}\label{prop necessary-error}
The HR scheme \eqref{eq:upU}-\eqref{eq:zstar} does not satisfy the fully-discrete
entropy inequality \eqref{eq:estentrint} without quadratic error term, whatever restrictive is the CFL condition.
\end{proposition}
\begin{proof}[Proof of Theorem~\ref{thentrfull_JSM}.]
Using \eqref{eq:kinups} and \eqref{eq:defdeltaM}, one has for $\xi\leq 0$
\begin{equation}\begin{array}{l}
	\dsp f^{n+1-}_i=M_i-\sigma_i\Bigl(\xi M_{i+1/2+}-\xi M_{i-1/2+}+(\xi-u_i)(M_{i-1/2+}-M_{i+1/2-})\Bigr)\\
	\dsp\hphantom{f^{n+1-}_i}=M_i-\sigma_i\Bigl(\xi(M_{i+1/2+}-M_{i+1/2-})+u_i(M_{i+1/2-}-M_{i-1/2+})\Bigr),
	\label{eq:fup-}
	\end{array}
\end{equation}
while for $\xi\geq 0$,
\begin{equation}\begin{array}{l}
	\dsp f^{n+1-}_i=M_i-\sigma_i\Bigl(\xi M_{i+1/2-}-\xi M_{i-1/2-}+(\xi-u_i)(M_{i-1/2+}-M_{i+1/2-})\Bigr)\\
	\dsp\hphantom{f^{n+1-}_i}=M_i-\sigma_i\Bigl(\xi(M_{i-1/2+}-M_{i-1/2-})+u_i(M_{i+1/2-}-M_{i-1/2+})\Bigr).
	\label{eq:fup+}
	\end{array}
\end{equation}
But because of \eqref{eq:domM_i+1/2}, one has $0\leq M_{i+1/2-},M_{i-1/2+}\leq M_i$.
Thus for all $\xi$ we get from \eqref{eq:fup-}-\eqref{eq:fup+} that
$f^{n+1-}_{i}\geq (1-\sigma_i(|u_i|+|\xi-u_i|))M_i\geq 0$ under the CFL condition \eqref{eq:CFLfull},
proving (i).

Then, we write the linearization of $H$ around the Maxwellian $M_i$
\begin{equation}
	H(f^{n+1-}_i,z_i)=H(M_i, z_i) +\partial_f H(M_i,z_i)\bigl(f^{n+1-}_i-M_i\bigr)+L_i,
	\label{eq:linHerr}
\end{equation}
where $L_i$ is a remainder. The linearized term $\partial_f H(M_i,z_i)\bigl(f^{n+1-}_i-M_i\bigr)$ in \eqref{eq:linHerr}
is nothing but the dissipation of the semi-discrete scheme,
that has been estimated in Proposition \ref{prop semi-discrete}.
Thus, multiplying \eqref{eq:hquasidiffdiss} by $-\sigma_i$, using the form \eqref{eq:kinH} of $H$
and the identity
\begin{equation}
	b^3-a^3-3a^2(b-a)=(b+2a)(b-a)^2,
	\label{eq:idremk}
\end{equation}
we get
\begin{equation}\begin{array}{l}
	\dsp \hphantom{\leq}\partial_f H(M_i,z_i)\bigl(f^{n+1-}_i-M_i\bigr)\\
	\dsp\leq -\sigma_i\bigl(\widetilde H_{i+1/2-}-\widetilde H_{i-1/2+}\bigr)\\
	\dsp +\sigma_i\xi\1_{\xi<0}\frac{g^2\pi^2}{6}\bigl(M_{i+1/2+}+2M_{i+1/2-}\bigr)
	\bigl(M_{i+1/2+}-M_{i+1/2-}\bigr)^2\\
	\dsp -\sigma_i\xi\1_{\xi>0}\frac{g^2\pi^2}{6}\bigl(M_{i-1/2-}+2M_{i-1/2+}\bigr)
	\bigl(M_{i-1/2-}-M_{i-1/2+}\bigr)^2.
	\end{array}
	\label{eq:lindiss}
\end{equation}
Then, using again the form of $H$ and \eqref{eq:idremk}, the quadratic term $L_i$ in \eqref{eq:linHerr}
can be expressed as
\begin{equation}
	L_i=\frac{g^2 \pi^2}{6} (2M_i+f^{n+1-}_i) (f^{n+1-}_i-M_i)^2.
	\label{eq:l1}
\end{equation}
We notice that in \eqref{eq:linHerr}, the time variation of the kinetic entropy $H$ is estimated by
a term linearized in $\Delta t^n$, that is itself estimated in \eqref{eq:lindiss} by a space integrated-conservative
difference and nonpositive dissipations, and nonnegative errors $L_i$ which are merely quadratic in $\Delta t^n$.
These errors $L_i$ do not vanish when the topography is constant, and moreover do not tend to zero
strongly for discontinuous data $U$.
The remainder of the argument is to prove that under a CFL condition, the quadratic terms $L_i$
are dominated by the dissipation terms, up to errors that are directly estimated in terms of the variations of the topography $z$.

Using \eqref{eq:fup-}, we have for any $\alpha>0$
\begin{equation}\begin{array}{l}
	\dsp L_i\leq  \frac{g^2 \pi^2}{6} \sigma_i^2(2M_i+f^{n+1-}_i)
	\Bigl((1+\alpha) \xi^2\bigl( M_{i+1/2+} - M_{i+1/2-}\bigr)^2\\
	\dsp\mkern 60mu +(1+1/\alpha)u_i^2\bigl( M_{i+1/2-} - M_{i-1/2+}\bigr)^2\Bigr),
	\quad\mbox{for all }\xi\leq 0,
	\label{eq:l1m}
	\end{array}
\end{equation}
and similarly with \eqref{eq:fup+}
\begin{equation}\begin{array}{l}
	\dsp L_i\leq  \frac{g^2 \pi^2}{6} \sigma_i^2(2M_i+f^{n+1-}_i)
	\Bigl((1+\alpha) \xi^2\bigl( M_{i-1/2+} - M_{i-1/2-}\bigr)^2\\
	\dsp\mkern 60mu +(1+1/\alpha)u_i^2\bigl( M_{i+1/2-} - M_{i-1/2+}\bigr)^2\Bigr),
	\quad\mbox{for all }\xi\geq 0.
	\label{eq:l1p}
	\end{array}
\end{equation}
Therefore, adding the estimates~\eqref{eq:linHerr}, \eqref{eq:lindiss},
\eqref{eq:l1m}, \eqref{eq:l1p} yields
\begin{equation}
	H(f_i^{n+1-},z_i)\leq H(M_i,z_i)-\sigma_i\Bigl(\widetilde{H}_{i+1/2-}
	-\widetilde{H}_{i-1/2+}\Bigr)+d_i,
	\label{eq:entrfullystat_orig}
\end{equation}
where
\begin{equation}\begin{array}{l}
	\dsp d_i = \sigma_i\xi\1_{\xi<0}\frac{g^2\pi^2}{6}\left(M_{i+1/2+}
	+2M_{i+1/2-}+(1+\alpha)\sigma_i\xi (2M_i+f^{n+1-}_i) \right)\\
	\dsp\mkern 350mu \times(M_{i+1/2+} -  M_{i+1/2-})^2\\
	\dsp\hphantom{d_i\leq} -\sigma_i\xi\1_{\xi>0}\frac{g^2\pi^2}{6}\left(M_{i-1/2-}
	+2M_{i-1/2+}-(1+\alpha)\sigma_i\xi(2M_i+f^{n+1-}_i)\right)\\
	\dsp\mkern 350mu \times(M_{i-1/2+} -M_{i-1/2-})^2\\
	\dsp\hphantom{d_i\leq} + \sigma_i^2 u_i^2 \frac{g^2\pi^2}{6}(1+1/\alpha) (2M_i+f^{n+1-}_i)
	(M_{i+1/2-} -  M_{i-1/2+})^2,
	\label{eq:d_i}
	\end{array}
\end{equation}
and $\alpha>0$ is an arbitrary parameter.
The first two lines in~\eqref{eq:d_i} are
generically nonpositive for $\sigma_i$ small enough (recall the bound \eqref{eq:vmax} on $\xi$),
whereas the third line is nonnegative.
\color{white}
\end{proof}
\color{black}
Before going further in the proof of Theorem~\ref{thentrfull_JSM}, i.e. upper bounding
$d_i$ by a sum of a dissipation term and an error, let us state a lemma, that gives
another expression for $d_i$, in which the nonpositive contributions appear clearly.
\begin{lemma}
The term $d_i$ from~\eqref{eq:d_i} can also be written
\begin{eqnarray}
d_i & = &
\sigma_i\xi\1_{\xi<0}\,\gamma_{i+1/2}^-
(M_{i+1/2+} -  M_{i+1/2-})^2\nonumber\\
&& - \sigma_i\xi\1_{\xi>0}\,\gamma_{i-1/2}^+
(M_{i-1/2+} -  M_{i-1/2-})^2 \nonumber\\
&& + \sigma_i^2 \frac{g^2\pi^2}{6} \biggl( (1+1/\alpha)u_i^2
(2M_i+f^{n+1-}_i)(M_{i+1/2-} -  M_{i-1/2+})^2 \nonumber\\
& & \mkern 100mu+ (1+\alpha)\xi^2\bigl(\1_{\xi <0}\,\mu_{i+1/2}^- + \1_{\xi > 0}\,\mu_{i-1/2}^+ \bigr)\biggr),
\label{eq:d_i_bis}
\end{eqnarray}
with
\begin{equation}\begin{array}{l}
	\dsp \gamma_{i+1/2}^-  =  \frac{g^2\pi^2}{6} \biggl(  \bigl(1- (1+\alpha)(\sigma_i\xi)^2\bigr) M_{i+1/2+}\\
	\dsp\hphantom{\gamma_{i+1/2}^-  =\qquad\quad}
	+ \Bigl(2+(1+\alpha)(\sigma_i\xi)^2 + 3(1+\alpha)\sigma_i\xi\Bigr) M_{i+1/2-} \biggr),\\
	\dsp\gamma_{i-1/2}^+  =  \frac{g^2\pi^2}{6} \biggl(  \bigl(1- (1+\alpha)(\sigma_i\xi)^2\bigr) M_{i-1/2-}\\
	\dsp\hphantom{\gamma_{i-1/2}^+  =\qquad\quad}
	+ \Bigl(2+(1+\alpha)(\sigma_i\xi)^2 - 3(1+\alpha)\sigma_i\xi\Bigr) M_{i-1/2+} \biggr),
	\label{eq:lambda+}
	\end{array}
\end{equation}
\begin{equation}\begin{array}{l}
	\dsp \mu_{i+1/2}^-  =  (M_{i+1/2+} -  M_{i+1/2-})^2\Bigl( 3(M_i - M_{i+1/2-})\\
	\dsp\mkern 280mu -\sigma_iu_i (M_{i+1/2-} - M_{i-1/2+})\Bigr),\\
	\dsp\mu_{i-1/2}^+  =  (M_{i-1/2+} -  M_{i-1/2-})^2\Bigl( 3(M_i -M_{i-1/2+})\\
	\dsp\mkern 280mu -\sigma_iu_i (M_{i+1/2-} - M_{i-1/2+})\Bigr).
	\label{eq:mu+}
	\end{array}
\end{equation}
\label{lemma:d_i_bis}
\end{lemma}
\begin{proof}[Proof of Lemma~\ref{lemma:d_i_bis}]
The expression \eqref{eq:fup-} of $f^{n+1-}_i$ for $\xi\leq 0$ allows
to precise the value of $d_i$ in \eqref{eq:d_i}, and gives for $\xi\leq 0$
\begin{eqnarray*}
&  & M_{i+1/2+} + 2M_{i+1/2-}+(1+\alpha)\sigma_i\xi (2M_i+f^{n+1-}_i)\\
& = & (1- (1+\alpha)(\sigma_i\xi)^2) M_{i+1/2+}
	+ (2+(1+\alpha)(\sigma_i\xi)^2) M_{i+1/2-}\\
& &  + (1+\alpha)\sigma_i \xi \left( 3M_i - \sigma_iu_i (M_{i+1/2-} - M_{i-1/2+})\right)\\
& = & (1- (1+\alpha)(\sigma_i\xi)^2) M_{i+1/2+}
	+ (2+(1+\alpha)(\sigma_i\xi)^2 + 3(1+\alpha)\sigma_i\xi) M_{i+1/2-}\\
& &  + (1+\alpha)\sigma_i \xi \left( 3(M_i - M_{i+1/2-})
- \sigma_iu_i (M_{i+1/2-} - M_{i-1/2+})\right).
\end{eqnarray*}
Using \eqref{eq:fup+} we obtain analogously for $\xi\geq 0$
\begin{eqnarray*}
&  & M_{i-1/2-} + 2M_{i-1/2+}-(1+\alpha)\sigma_i\xi (2M_i+f^{n+1-}_i)\\
& = & (1- (1+\alpha)(\sigma_i\xi)^2) M_{i-1/2-}
	+ (2+(1+\alpha)(\sigma_i\xi)^2) M_{i-1/2+}\\
& &  - (1+\alpha)\sigma_i \xi \left( 3M_i - \sigma_iu_i (M_{i+1/2-} - M_{i-1/2+})\right)\\
& = & (1- (1+\alpha)(\sigma_i\xi)^2) M_{i-1/2-}
	+ (2+(1+\alpha)(\sigma_i\xi)^2 - 3(1+\alpha)\sigma_i\xi) M_{i-1/2+}\\
& &  - (1+\alpha)\sigma_i \xi \left( 3(M_i - M_{i-1/2+})
- \sigma_iu_i (M_{i+1/2-} - M_{i-1/2+})\right).
\end{eqnarray*}
These expressions yield the formulas \eqref{eq:d_i_bis}-\eqref{eq:mu+}.
\end{proof}

\begin{proof}[Continuation of the proof of Theorem~\ref{thentrfull_JSM}]
One would like the first two lines of \eqref{eq:d_i_bis} to be nonpositive.
In order to get nonnegative coefficients $\gamma_{i+1/2}^-$, $\gamma_{i-1/2}^+$
in \eqref{eq:d_i_bis}, it is enough that
\begin{equation}
	1-(1+\alpha)(\sigma_i|\xi|)^2\geq 0,\quad 2+(1+\alpha)(\sigma_i|\xi|)^2-3(1+\alpha)\sigma_i|\xi|\geq 0,
	\label{eq:condpposgamma}
\end{equation}
for all $\xi$ in the supports of $M_{i-1}$, $M_i$, $M_{i+1}$.
But since both expressions in \eqref{eq:condpposgamma} are decreasing with respect
to $|\xi|$ for $\sigma_i|\xi|\leq 1$ and because of the CFL condition~\eqref{eq:CFLfull},
they are lower bounded respectively by
\begin{equation}
	1-(1+\alpha)\beta^2,\quad 2+(1+\alpha)\beta^2-3(1+\alpha)\beta.
	\label{eq:lowerbeta}
\end{equation}
But since $\beta<1$, one can choose $\alpha>0$ such that
\begin{equation}
	1+\alpha<\frac{2}{\beta(3-\beta)},
	\label{eq:alphabeta}
\end{equation}
and then the coefficients \eqref{eq:lowerbeta} are positive, and
$\gamma_{i+1/2}^-,\gamma_{i-1/2}^+\geq 0$. 
We denote
\begin{equation}
	c_{\alpha,\beta}=\min\Bigl(1-(1+\alpha)\beta^2
	,2+(1+\alpha)\beta^2-3(1+\alpha)\beta\Bigr)>0.
	\label{eq:defc}
\end{equation}
Then we have
\begin{equation}
	\1_{\xi<0}\gamma_{i+1/2}^-\geq \1_{\xi<0}\frac{g^2\pi^2}{6}c_{\alpha,\beta}(M_{i+1/2+}+M_{i+1/2-}),
	\label{eq:lowgamma-}
\end{equation}
and
\begin{equation}
	\1_{\xi>0}\gamma_{i-1/2}^+\geq \1_{\xi>0}\frac{g^2\pi^2}{6}c_{\alpha,\beta}(M_{i-1/2-}+M_{i-1/2+}).
	\label{eq:lowgamma+}
\end{equation}
Next we write using \eqref{eq:fup-}, \eqref{eq:fup+} and \eqref{eq:domM_i+1/2}
\begin{equation}\begin{array}{l}
	\dsp\hphantom{\leq}\, 2M_i+f_i^{n+1-}\\
	\dsp \leq 3M_i -\sigma_i\xi\1_{\xi<0}(M_{i+1/2+}-M_{i+1/2-})_+\\
	\dsp\hphantom{\leq}\,+ \sigma_i\xi\1_{\xi>0}(M_{i-1/2-}-M_{i-1/2+})_+
	+\sigma_i|u_i||M_{i+1/2-}-M_{i-1/2+}|\\
	\dsp\leq 4M_i -\sigma_i\xi\1_{\xi<0}(M_{i+1/2+}-M_{i+1/2-})_+
	+ \sigma_i\xi\1_{\xi>0}(M_{i-1/2-}-M_{i-1/2+})_+.
	\label{eq:upper2M+f}
	\end{array}
\end{equation}
We can estimate the first quadratic error term from \eqref{eq:d_i_bis} as
\begin{equation}\begin{array}{l}
	\dsp\hphantom{\leq}\, (2M_i+f^{n+1-}_i)(M_{i+1/2-} -  M_{i-1/2+})^2\\
	\dsp\leq 4M_i(M_{i+1/2-} -  M_{i-1/2+})^2\\
	\dsp\hphantom{\leq}\,
	-\sigma_i\xi\1_{\xi<0}M_i|M_{i+1/2+}-M_{i+1/2-}||M_{i+1/2-} -  M_{i-1/2+}|\\
	\dsp\hphantom{\leq}\,
	+\sigma_i\xi\1_{\xi>0}M_i|M_{i-1/2-}-M_{i-1/2+}||M_{i+1/2-} -  M_{i-1/2+}|.
	\label{eq:estquadr1}
	\end{array}
\end{equation}
Finally we estimate
\begin{equation}\begin{array}{l}
	\dsp\hphantom{\leq}\ |\mu_{i+1/2}^-|\\
	\dsp\leq 4(M_{i+1/2+} -  M_{i+1/2-})^2\bigl(|M_i-M_{i+1/2-}|+|M_i-M_{i-1/2+}|\bigr)\\
	\dsp\leq 2|M_{i+1/2+} - M_{i+1/2-}|\Bigl(\epsilon(M_{i+1/2+} -  M_{i+1/2-})^2\\
	\dsp\mkern 100mu +\epsilon^{-1}\bigl(|M_i-M_{i+1/2-}|+|M_i-M_{i-1/2+}|\bigr)^2\Bigr)\\
	\dsp\leq 2\epsilon(M_{i+1/2+} + M_{i+1/2-})(M_{i+1/2+} -  M_{i+1/2-})^2\\
	\dsp\hphantom{\leq} +4\epsilon^{-1}M_i|M_{i+1/2+} - M_{i+1/2-}|\bigl(|M_i-M_{i+1/2-}|+|M_i-M_{i-1/2+}|\bigr),
	\label{eq:estmu-}
	\end{array}
\end{equation}
and similarly
\begin{equation}\begin{array}{l}
	\dsp\hphantom{\leq}\ |\mu_{i-1/2}^+|\\
	\dsp\leq 2\epsilon(M_{i-1/2-} + M_{i-1/2+})(M_{i-1/2+} -  M_{i-1/2-})^2\\
	\dsp\hphantom{\leq} +4\epsilon^{-1}M_i|M_{i-1/2+} - M_{i-1/2-}|\bigl(|M_i-M_{i+1/2-}|+|M_i-M_{i-1/2+}|\bigr),
	\label{eq:estmu+}
	\end{array}
\end{equation}
where $\epsilon>0$ is arbitrary. Putting together in \eqref{eq:d_i_bis} the estimates
\eqref{eq:lowgamma-}, \eqref{eq:lowgamma+},
\eqref{eq:estmu-}, \eqref{eq:estmu+}, we get
\begin{equation}\begin{array}{l}
	\dsp d_i\leq\sigma_i\xi\1_{\xi<0}\frac{g^2\pi^2}{6}\bigl(c_{\alpha,\beta}-2\epsilon(1+\alpha)\sigma_i|\xi|\bigr)\\
	\dsp\mkern 160mu\times(M_{i+1/2+}+M_{i+1/2-})(M_{i+1/2+} -  M_{i+1/2-})^2\\
	\dsp\hphantom{d_i\leq}
	- \sigma_i\xi\1_{\xi>0}\frac{g^2\pi^2}{6}\bigl(c_{\alpha,\beta}-2\epsilon(1+\alpha)\sigma_i|\xi|\bigr)\\
	\dsp\mkern 160mu\times(M_{i-1/2-}+M_{i-1/2+})(M_{i-1/2+} -  M_{i-1/2-})^2\\
	\dsp\hphantom{d_i\leq}
	+\sigma_i^2 \frac{g^2\pi^2}{6} \biggl( (1+1/\alpha)u_i^2
	(2M_i+f^{n+1-}_i)(M_{i+1/2-} -  M_{i-1/2+})^2\\
	\dsp\mkern 60mu+ 4\epsilon^{-1}(1+\alpha)\xi^2M_i\bigl(|M_i-M_{i+1/2-}|+|M_i-M_{i-1/2+}|\bigr)\\
	\dsp\mkern 80mu\times\bigl(\1_{\xi <0}|M_{i+1/2+} - M_{i+1/2-}| + \1_{\xi > 0}|M_{i-1/2+} - M_{i-1/2-}| \bigr)\biggr).
	\label{eq:d_i_est}
	\end{array}
\end{equation}
We set
\begin{equation}
	\nu_\beta^0=c_{\alpha,\beta}-2\epsilon(1+\alpha)\beta,
	\label{eq:valnubeta}
\end{equation}
which is positive if $\epsilon$ is taken small enough (recall that $\alpha>0$ has been chosen
so as to satisfy \eqref{eq:alphabeta}, and hence depends only on $\beta$).
Then using \eqref{eq:entrfullystat_orig} and \eqref{eq:d_i_est}, the two first lines
in the right-hand side of \eqref{eq:d_i_est} give a dissipation as stated in \eqref{eq:entrfullystat},
while the last lines give an error. From \eqref{eq:d_i_est} and \eqref{eq:estquadr1},
for $\xi<0$ the typical error terms take the form
\begin{equation}\begin{array}{l}
	\dsp\hphantom{\leq}\ M_i|M_{i+1/2+}-M_{i+1/2-}||M_i -  M_{i-1/2+}|\\
	\dsp=\bigl(\1_{M_i\leq M_{i+1/2+}}+\1_{M_i>M_{i+1/2+}}\bigr)M_i|M_{i+1/2+}-M_{i+1/2-}||M_i -  M_{i-1/2+}|\\
	\dsp\leq\1_{M_i\leq M_{i+1/2+}}M_i\Bigl(\epsilon_2|M_{i+1/2+}-M_{i+1/2-}|^2+\epsilon_2^{-1}|M_i -  M_{i-1/2+}|^2\Bigr)\\
	\dsp\hphantom{\leq}
	+\1_{M_i>M_{i+1/2+}}\Bigl(M_{i+1/2-}|M_{i+1/2+}-M_{i+1/2-}||M_i -  M_{i-1/2+}|\\
	\dsp\mkern 60mu+|M_i -  M_{i+1/2-}||M_{i+1/2+}-M_{i+1/2-}||M_i -  M_{i-1/2+}|\Bigr)\\
	\dsp\leq\epsilon_2M_{i+1/2+}|M_{i+1/2+}-M_{i+1/2-}|^2
	+\epsilon_2^{-1}M_i|M_i -  M_{i-1/2+}|^2\\
	\dsp\hphantom{\leq}+M_{i+1/2-}\Bigl(\epsilon_2|M_{i+1/2+}-M_{i+1/2-}|^2+\epsilon_2^{-1}|M_i -  M_{i-1/2+}|^2\Bigr)\\
	\dsp\hphantom{\leq}+M_i|M_i -  M_{i+1/2-}||M_i -  M_{i-1/2+}|\\
	\dsp\leq\epsilon_2\bigl(M_{i+1/2+}+M_{i+1/2-}\bigr)|M_{i+1/2+}-M_{i+1/2-}|^2\\
	\dsp\hphantom{\leq}
	+3\epsilon_2^{-1}M_i|M_i -  M_{i-1/2+}|^2+\epsilon_2M_i|M_i -  M_{i+1/2-}|^2.
	\label{eq:estcross}
	\end{array}
\end{equation}
The term proportional to $\epsilon_2$ can therefore be absorbed by $\nu_\beta^0$.
Since a similar estimate holds for $\xi>0$, diminishing slightly $\nu_\beta^0$
by something proportional to $\epsilon_2$ (taken small enough), we get a coefficient $\nu_\beta>0$.
The only remaining error terms finally take the form stated in the last line of \eqref{eq:entrfullystat}.
This completes the proof of (ii) in Theorem \ref{thentrfull_JSM}.
\end{proof}
\begin{remark}
Consider the situation when for some $i_0$ one has
$$u_{i_0-1} = u_{i_0} =u_{i_0+1}\neq 0 \text{ 
and } h_{i_0-1}+z_{i_0-1} = h_{i_0}+z_{i_0} =h_{i_0+1}+z_{i_0+1},$$
with $z_{i_0-1} \neq z_{i_0}$ or $z_{i_0} \neq z_{i_0+1}$.
Then by \eqref{eq:lrstates}, \eqref{eq:hlr}, the reconstructed states satisfy
$U_{i+1/2-}=U_{i+1/2+}$ for $i=i_0-1,i_0$.
We observe that then, in the formula~\eqref{eq:d_i} for $d_i$, the dissipative terms vanish
for $i=i_0$, for all $\xi$. Thus $d_{i_0}\geq 0$ and $\int d_{i_0}(\xi) d\xi > 0$, 
which means that the extra term $d_i$ in \eqref{eq:entrfullystat_orig} gives 
a dissipation with the wrong sign, in agreement with Proposition \ref{prop necessary-error}.
\end{remark}
\begin{proof}[Proof of Proposition~\ref{prop necessary-error}]
It has been proved in \cite{ABBKP} that the semi-discrete HR scheme (limit $\sigma_i\rightarrow 0$)
satisfies the entropy inequality without error term. Here we prove that the fully-discrete scheme does not,
whatever restrictive is the CFL condition. This result holds for an arbitrary numerical flux $\F$ taken for the
homogeneous Saint-Venant system. The argument is as follows.

Consider the local dissipation
\begin{equation}
	{\cal D}_i^n=\eta(U_i^{n+1})+gz_ih_i^{n+1}-\eta(U_i)-gz_ih_i
	+\sigma_i\Bigl(\widetilde G_{i+1/2}-\widetilde G_{i-1/2}\Bigr),
	\label{eq:locdiss}
\end{equation}
where $U^{n+1}_i$ is given by~\eqref{eq:upU},
$F_{i+1/2\pm}$ are defined by \eqref{eq:HRflux}-\eqref{eq:zstar}, and
\begin{equation}
	\widetilde G_{i+1/2}={\cal G}(U_{i+1/2-},U_{i+1/2+})+gz_{i+1/2}\F^0(U_{i+1/2-},U_{i+1/2+}),
	\label{eq:numentrfluxbis}
\end{equation}
where $\cal G$ is the numerical entropy flux associated to $\F$,
and $\F^0$ is the first (density) component of $\F$.
Then, taking into account that $h_i^{n+1}=h_i-\sigma_i(\F^0(U_{i+1/2-},U_{i+1/2+})-\F^0(U_{i-1/2-},U_{i-1/2+}))$, one has
\begin{equation}\begin{array}{l}
	\dsp \frac{{\cal D}_i^n}{\sigma_i}=
	\frac{\eta\left(U_i-\sigma_i(F_{i+1/2-}-F_{i-1/2+})\right)-\eta(U_i)}{\sigma_i}\\
	\dsp\hphantom{\frac{{\cal D}_i^n}{\sigma_i}=}
	-gz_i\bigl(\F^0(U_{i+1/2-},U_{i+1/2+})-\F^0(U_{i-1/2-},U_{i-1/2+})\bigr)\vphantom{\Biggl|}\\
	\dsp\hphantom{\frac{{\cal D}_i^n}{\sigma_i}=}
	+\widetilde G_{i+1/2}-\widetilde G_{i-1/2}\vphantom{\Bigl|}.
	\label{eq:locdissnorm}
	\end{array}
\end{equation}
The entropy $\eta$ being strictly convex, the function
\begin{equation}
	\sigma_i\mapsto \eta\left(U_i-\sigma_i(F_{i+1/2-}-F_{i-1/2+})\right)
	\label{eq:convexsig}
\end{equation}
is convex, and strictly convex if
\begin{equation}
	F_{i+1/2-}-F_{i-1/2+}\not =0.
	\label{eq:convexcond}
\end{equation}
Assuming that this condition holds, we get that the right-hand side of \eqref{eq:locdissnorm}
is strictly increasing with respect to $\sigma_i$. In particular, it will be strictly positive
if the limit as $\sigma_i\rightarrow 0$ of this quantity vanishes. This limit is nothing else than
the dissipation of the semi-discrete scheme
\begin{equation}\begin{array}{l}
	\dsp-\eta'(U_i)(F_{i+1/2-}-F_{i-1/2+})+\widetilde G_{i+1/2}-\widetilde G_{i-1/2}\\
	\dsp-gz_i\bigl(\F^0(U_{i+1/2-},U_{i+1/2+})-\F^0(U_{i-1/2-},U_{i-1/2+})\bigr)\vphantom{\Bigl|}.
	\end{array}
	\label{eq:disssemi}
\end{equation}
Consider data such that
\begin{equation}
	U_i=U_l,\ z_i=z_l\mbox{ for }i\leq i_0,\quad U_i=U_r,\ z_i=z_r\mbox{ for }i>i_0,
	\label{eq:datanondiss}
\end{equation}
for left and right states $U_l=(h_l,h_lu_l)$, $U_r=(h_r,h_ru_r)$ such that
\begin{equation}
	u_l=u_r\not=0,\qquad h_l+z_l=h_r+z_r,\qquad z_r-z_l>0.
	\label{eq:datafalseeq}
\end{equation}
Then one checks easily that \eqref{eq:convexcond} holds for $i=i_0$, and that
\eqref{eq:disssemi} vanishes for all $i$. Therefore, ${\cal D}_{i_0}^n>0$, which proves the claim.
\end{proof}
The following lemma establishes a kind of $L^2-$Lipschitz dependency of the Maxwellian
with respect to $U$, that allows to estimate the integral of the error terms in \eqref{eq:entrfullystat}.
Note that the Maxwellian~\eqref{eq:kinmaxw} is only $1/2$-H\"older continuous at fixed $\xi$.
\begin{lemma}\label{lemma lipestimate}
Let $U_k=(h_k,h_ku_k)$ for $k=1,2,3$ with $h_k\geq 0$.
Then
\begin{equation}\begin{array}{l}
	\dsp\hphantom{\leq}\int_\r M(U_1,\xi)\Bigl(M(U_1,\xi)-M(U_2,\xi)\Bigr)^2d\xi\\
	\dsp\leq \frac{3}{g^2\pi^2}\Bigl(g(h_2-h_1)^2+\min(h_1,h_2)(u_2-u_1)^2\Bigr),
	\end{array}
	\label{eq:lipest}
\end{equation}
and
\begin{equation}\begin{array}{l}
	\dsp\hphantom{\leq}\int_\r M(U_3,\xi)\Bigl(M(U_1,\xi)-M(U_2,\xi)\Bigr)^2d\xi\\
	\dsp\leq \frac{6}{g^2\pi^2}\Bigl(g(h_3-h_1)^2+g(h_3-h_2)^2\\
	\dsp\mkern 20mu
	+\min(h_1,h_3)(u_3-u_1)^2+\min(h_2,h_3)(u_3-u_2)^2\Bigr).
	\end{array}
	\label{eq:lipest3}
\end{equation}
\end{lemma}
\begin{proof}
One has
\begin{equation}\begin{array}{l}
	\dsp\hphantom{\leq}\int_\r M(U_1,\xi)\Bigl(M(U_1,\xi)-M(U_2,\xi)\Bigr)^2d\xi\\
	\dsp\leq\frac{1}{2}\int_\r \Bigl(2M(U_1,\xi)+M(U_2,\xi)\Bigr)\Bigl(M(U_1,\xi)-M(U_2,\xi)\Bigr)^2d\xi\\
	\dsp=\frac{3}{g^2\pi^2}\int_\r\Bigl(H_0(M(U_2,\xi),\xi)-H_0(M(U_1,\xi),\xi)\\
	\dsp\mkern 100mu-\partial_f H_0(M(U_1,\xi),\xi)(M(U_2,\xi)-M(U_1,\xi))\Bigr)d\xi\\
	\dsp\leq\frac{3}{g^2\pi^2}\int_\r\Bigl(H_0(M(U_2,\xi),\xi)-H_0(M(U_1,\xi),\xi)\\
	\dsp\mkern 100mu-\eta'(U_1)\kxi(M(U_2,\xi)-M(U_1,\xi))\Bigr)d\xi\\
	\dsp=\frac{3}{g^2\pi^2}\Bigl(\eta(U_2)-\eta(U_1)-\eta'(U_1)(U_2-U_1)\Bigr)\\
	\dsp=\frac{3}{g^2\pi^2}\Bigl(g\frac{(h_2-h_1)^2}{2}+h_2\frac{(u_2-u_1)^2}{2}\Bigr).
	\end{array}
	\label{eq:lip1}
\end{equation}
We can also estimate $M(U_1,\xi)$ by $M(U_1,\xi)+2M(U_2,\xi)$, giving the same estimate
as \eqref{eq:lip1} with $U_1$ and $U_2$ exchanged and with an extra factor $2$. This proves \eqref{eq:lipest}.
Then, denoting $M_k\equiv M(U_k,\xi)$, according to the Minkowsky inequality,
\begin{equation}\begin{array}{l}
	\dsp\hphantom{\leq} \left(\int_\r M_3\bigl(M_1-M_2\bigr)^2d\xi\right)^{1/2}\\
	\dsp\leq\left(\int_\r M_3\bigl(M_1-M_3\bigr)^2d\xi\right)^{1/2}
	+\left(\int_\r M_3\bigl(M_3-M_2\bigr)^2d\xi\right)^{1/2},
	\label{eq:estmink}
	\end{array}
\end{equation}
Using \eqref{eq:lipest}, we obtain \eqref{eq:lipest3}.
\end{proof}

\section{Conclusion}
We have established that the unmodified hydrostatic reconstruction scheme for the Saint Venant
system with topography satisfies a fully discrete entropy inequality \eqref{eq:estentrint} with error term,
in the case when the homogeneous numerical flux is the kinetic one with the Maxwellian \eqref{eq:kinmaxw}.
This inequality is obtained as the integral with respect to the kinetic variable $\xi$ of
a discrete kinetic entropy inequality \eqref{eq:entrfullystat} with error term.
These error terms are not present in the case when the entropy dissipation is linearized
with respect to the timestep $\Delta t$ (or equivalently in the semi-discrete case). They come from
the less dissipative nature of explicit schemes with respect to their semi-discrete versions,
as appears clearly in the case without topography of Lemma \ref{lemma:seminotopo}.
In the case with topography, the identity \eqref{eq:linHerr} enables to write the entropy
dissipation as a sum of the one for the semi-discrete scheme plus an error term $L_i$
which has the wrong sign, and which is merely quadratic in $\Delta t$, see \eqref{eq:l1}-\eqref{eq:l1p}.
In general, the second-order in $\Delta t$ terms appearing in the entropy dissipation
are dominated (under a CFL condition) by the linear in $\Delta t$ dissipation terms.
However, since here we have a well-balanced scheme, these first-order terms degenerate
at the steady states at rest, and cannot dominate the second-order terms.
This is why error terms remain in \eqref{eq:estentrint}.
Nevertheless, these errors are estimated in the square of the topography jumps,
and do not involve jumps in the unknown $U$, that would not be small in the case of shocks.
This property enables to proceed with a proof of convergence of the scheme, that will
be provided in a forthcoming paper.

An open problem that remains however is to establish the fully discrete entropy inequality with error
\eqref{eq:estentrint} for a HR scheme with general (non kinetic)
homogeneous numerical flux $\F$ satisfying a fully discrete entropy inequality.

\section*{Acknowledgments}
The authors wish to express their warm thanks to Carlos Par\'es Madro\~nal for many fruitful discussions.
This work has been partially funded by the ANR contract ANR-11-BS01-0016 LANDQUAKES.


\begin{thebibliography}{}

\end{thebibliography}


\begin{thebibliography}{99}

\bibitem{ABBKP} E. Audusse, F. Bouchut, M.-O. Bristeau, R. Klein, B. Perthame,
\emph{A fast and stable well-balanced scheme with hydrostatic reconstruction for shallow water flows},
SIAM J. Sci. Comp. 25 (2004), 2050-2065.

\bibitem{bristeau} E.~Audusse, M.-O. Bristeau,
\emph{A well-balanced positivity preserving second-order scheme for
{S}hallow {W}ater flows on unstructured meshes},
J. Comput. Phys. 206 (2005), 311-333.

\bibitem{ABPelSM} E.~Audusse, M.-O. Bristeau, M.~Pelanti, J.~Sainte-Marie,
\emph{Approximation of the hydrostatic Navier-Stokes system for density
stratified flows by a multilayer model. Kinetic interpretation and numerical validation},
J. Comp. Phys. 230 (2011), 3453-3478.

\bibitem{kinetic-RR} E.~Audusse, M.-O. Bristeau, B.~Perthame,
\emph{Kinetic schemes for Saint Venant equations with source terms on unstructured grids},
Technical Report 3989, INRIA, Unit\'e de recherche de Rocquencourt, France, 2000.
\newblock http://www.inria.fr/rrrt/rr-3989.html.

\bibitem{ABPerthSM} E.~Audusse, M.-O. Bristeau, B.~Perthame, J.~Sainte-Marie,
\emph{A multilayer Saint-Venant system with mass exchanges for Shallow Water flows.
Derivation and numerical validation}, ESAIM: M2AN 45 (2011), 169-200.

\bibitem{Ber} F. Berthelin,
\emph{Convergence of flux vector splitting schemes with single entropy inequality for hyperbolic systems of conservation laws},
Numer. Math. 99 (2005), 585-604.

\bibitem{BB} F. Berthelin, F. Bouchut,
\emph{Relaxation to isentropic gas dynamics for a BGK system with single kinetic entropy},
Meth. and Appl. of Analysis 9 (2002), 313-327.


\bibitem{BGKcons} F. Bouchut,
\emph{Construction of BGK models with a family of kinetic entropies for a given system of conservation laws},
J. Stat. Phys. 95 (1999), 113-170.

\bibitem{FVS} F. Bouchut,
\emph{Entropy satisfying flux vector splittings and kinetic BGK models},
Numer. Math. 94 (2003), 623-672.

\bibitem{bouchut_book} F.~Bouchut,
\emph{Nonlinear stability of finite volume methods for hyperbolic
conservation laws, and well-balanced schemes for sources},
Birkh\"auser, 2004.

\bibitem{BM} F. Bouchut, T. Morales,
\emph{A subsonic-well-balanced reconstruction scheme for shallow water flows},
Siam J. Numer. Anal. 48 (2010), 1733-1758.

\bibitem{BZ} F. Bouchut, V. Zeitlin,
\emph{A robust well-balanced scheme for multi-layer shallow water equations},
Discrete and Continuous Dynamical Systems - Series B, 13 (2010), 739-758.

\bibitem{BGSM} M.-O. Bristeau, N.~Goutal, J.~Sainte-Marie,
\emph{Numerical simulations of a non-hydrostatic {S}hallow {W}ater model},
Computers \& Fluids 47 (2011), 51-64.

\bibitem{CPP} M.J. Castro, A. Pardo Milan\'es, C. Par\'es,
\emph{Well-balanced numerical schemes based on a generalized hydrostatic reconstruction technique},
Math. Models Methods Appl. Sci. 17 (2007), 2055-2113.

\bibitem{CSS} F. Coquel, K. Saleh, N. Seguin,
\emph{A robust and entropy-satisfying numerical scheme for fluid flows in discontinuous nozzles},
Math. Models Meth. Appl. Sci. 24 (2014), 2043.

\bibitem{perthame_coron} F.~Coron, B.~Perthame,
\emph{Numerical passage from kinetic to fluid equations},
SIAM J. Numer. Anal. 28 (1991), 26-42.


\bibitem{Gbook} L. Gosse,
\emph{Computing qualitatively correct approximations of balance laws. Exponential-fit, well-balanced and asymptotic-preserving},
SIMAI Springer Series, 2. Springer, Milan, 2013.

\bibitem{GSM} N.~Goutal, J.~Sainte-Marie,
\emph{A kinetic interpretation of the section-averaged Saint-Venant system
for natural river hydraulics},
Int. J. Numer. Meth. Fluids 67 (2011), 914-938.

\bibitem{J} S. Jin,
\emph{Asymptotic preserving (AP) schemes for multiscale kinetic and hyperbolic equations: a review},
Lecture Notes for Summer School on ``Methods and Models of Kinetic Theory'' (M\&MKT), Porto Ercole (Grosseto, Italy), June 2010.
Rivista di Matematica della Universite di Parma 3 (2012), 177-216. 

\bibitem{NX} S. Noelle, Y. Xing, C.-W. Shu,
\emph{High-order well-balanced finite volume WENO schemes for shallow water equation with moving water},
J. Comput. Phys. 226 (2007), 29-58.

\bibitem{Perthame90} B. Perthame,
\emph{Boltzmann type schemes for gas-dynamics and the entropy property},
SIAM J. Numer. Anal. 27 (1990), 1405-1421.

\bibitem{Perthame92} B. Perthame,
\emph{2nd-order {B}oltzmann schemes for compressible {E}uler equations in one and 2 space dimensions},
SIAM J. Numer. Anal. 29 (1992), 1-19.

\bibitem{Perthame} B. Perthame,
\emph{Kinetic formulation of conservation laws},
Oxford Lecture Series in Mathematics and its Applications, 21.
Oxford University Press, Oxford, 2002.

\bibitem{PS} B. Perthame, C. Simeoni,
\emph{A kinetic scheme for the Saint Venant system with a source term},
Calcolo 38 (2001), 201-231.

\bibitem{XS} Y. Xing, C.-W. Shu,
\emph{A survey of high order schemes for the shallow water equations},
Journal of Mathematical Study 47 (2014), 221-249. 

\end{thebibliography}

\bibliographystyle{abbrv}

\end{document}